\documentclass[a4paper,10pt]{amsart}
\usepackage{a4wide}
\usepackage[latin1]{inputenc}
\usepackage[T1]{fontenc}
\usepackage{amsmath,amssymb,amsthm,stmaryrd}
\usepackage{mathrsfs}
\usepackage{esint}
\usepackage{color}

\newcommand{\ud}[0]{\,\mathrm{d}}

\newcommand{\dist}[0]{\operatorname{dist}}
\newcommand{\eps}[0]{\varepsilon}

\newcommand{\abs}[1]{|#1|}
\newcommand{\Babs}[1]{\Big|#1\Big|}

\newcommand{\Norm}[2]{\|#1\|_{#2}}

\newcommand{\BNorm}[2]{\Big\|#1\Big\|_{#2}}
\newcommand{\pair}[2]{\langle #1,#2 \rangle}
\newcommand{\Bpair}[2]{\Big\langle #1,#2 \Big\rangle}
\newcommand{\ave}[1]{\langle #1\rangle}

\newcommand{\BMO}[0]{\operatorname{BMO}}
\newcommand{\supp}[0]{\operatorname{supp}}

\newcommand{\R}{\mathbb{R}}

\newcommand{\Z}{\mathbb{Z}}

\swapnumbers
\numberwithin{equation}{section}

\theoremstyle{plain}
\newtheorem{theorem}[equation]{Theorem}
\newtheorem{proposition}[equation]{Proposition}

\newtheorem{lemma}[equation]{Lemma}

\theoremstyle{definition}
\newtheorem{definition}[equation]{Definition}

\theoremstyle{notation}

\theoremstyle{remark}
\newtheorem{remark}[equation]{Remark}

\makeatletter
\@namedef{subjclassname@2020}{%
  \textup{2020} Mathematics Subject Classification}
\makeatother

\title{The local $Tb$ theorem with rough test functions}
\author{Tuomas Hyt\"onen}
\address{Department of Mathematics and Statistics, University of Helsinki, P.O.B. 68, FI-00014 Helsinki, Finland}
\email{tuomas.hytonen@helsinki.fi}
\author{Fedor Nazarov}
\address{Kent State University, Department of Mathematical Sciences, Summit Street, Kent OH 44242, United States}
\email{nazarov@math.kent.edu}


\begin{document}

\maketitle

\begin{abstract}
We prove a local $Tb$ theorem under close to minimal (up to certain `buffering') integrability assumptions, conjectured by S.~Hofmann (El Escorial, 2008): Every cube is assumed to support two non-degenerate functions $b^1_Q\in L^p$ and $b^2_Q\in L^q$ such that $1_{2Q}Tb^1_Q\in L^{q'}$ and $1_{2Q}T^*b^2_Q\in L^{p'}$, with appropriate uniformity and scaling of the norms. This is sufficient for the $L^2$-boundedness of the Calder\'on--Zygmund operator $T$, for any $p,q\in(1,\infty)$, a result previously unknown for simultaneously small values of $p$ and $q$. We obtain this as a corollary of a local $Tb$ theorem for the maximal truncations $T_{\#}$ and $(T^*)_{\#}$: for the $L^2$-boundedness of $T$, it suffices that $1_Q T_{\#}b^1_Q$ and $1_Q (T^*)_{\#}b^2_Q$ be uniformly in $L^0$. The proof builds on the technique of suppressed operators from the quantitative Vitushkin conjecture due to Nazarov--Treil--Volberg.
\end{abstract}

\setcounter{tocdepth}{1}
\tableofcontents

\section{Introduction}

The first $Tb$ theorems were proved by David, Journ\'e and Semmes \cite{DJS}, and McIntosh and Meyer \cite{McMe}. Their idea was to characterize the $L^2$-boundedness of a singular integral operator $T$,
\begin{equation*}
  Tf(x)=\int_{\R^d}K(x,y)f(y)\ud y,
\end{equation*}
where $K$ is a standard Calder\'on--Zygmund kernel, by its (and its adjoint's) action on just one sufficiently non-degenerate function $b$. Thus, they generalized the celebrated $T1$ theorem of David and Journ\'e \cite{DJ:T1}, where this function was required to be $b\equiv 1$.

Another significant step in this type of characterizations was taken by Christ \cite{Christ:90}, who introduced the idea of a \emph{local $Tb$ theorem}. Rather than testing $T$ and $T^*$ on two globally well-behaved (and hence not so easy to find) functions $b_1$ and $b_2$, the operators can be tested against a family of local functions $b_Q^1$ and $b^2_Q$, indexed by the cubes (say) $Q$ on which they are supported, each of which is only required to satisfy a set of conditions on its `own' cube.

Besides necessary non-degeneracy requirements, Christ's assumptions on his test functions consisted of the uniform boundedness $b^1_Q,b^2_Q,Tb^1_Q,T^*b^2_Q\in L^\infty$. Weakening these conditions has been a topic of subsequent developments. Nazarov, Treil and Volberg \cite{NTV:Duke} showed (even in the more general context of possibly non-doubling underlying measures in place of the Lebesgue measure) that it suffices to have $b^1_Q,b^2_Q\in L^\infty$ and $T b^1_Q,T^*b^2_Q\in\BMO$, uniformly in $Q$. On the other hand, for certain dyadic model operators, Auscher, Hofmann, Muscalu, Tao and Thiele \cite{AHMTT} were able to relax these conditions to a substantially lower degree of integrability, namely
\begin{equation}\label{eq:Hofmann}
  b^1_Q\in L^p,\quad b^2_Q\in L^q,\quad 1_QTb^1_Q\in L^{q'},\quad 1_QT^*b^2_Q\in L^{p'}
\end{equation}
for any $p,q\in(1,\infty)$, where the different $L^r$ norms are appropriately scaled relative to $\abs{Q}$. The question then became whether these testing conditions for the model case also suffice for genuine singular integral operators. This was the first of the four open problems on local $Tb$ theorems formulated by Hofmann during his plenary lectures at the International Conference on Harmonic Analysis and P.D.E. in El Escorial, 2008; it was motivated by possible applications to layer potentials and to free boundary theory (see \cite[Section 3.3.1]{Hof:Escorial}).

Towards the solution of Hofmann's problem, the following developments have taken place. First, Hofmann's \cite{Hof:unpubl} positive result concerning the case $b^i_Q\in L^2$, $T_i b^i_Q\in L^{2+\eps}$. Next, Auscher and Yang's \cite{AY} elimination of the $\eps>0$ by a reduction to the dyadic case. In fact, they settled the result for all `large enough' pairs of exponents $p,q\in(1,\infty)$, namely, subject to the sub-duality condition $1/p+1/q\leq 1$. Finally, Auscher and Routin's \cite{AR} work on general pairs $p,q\in(1,\infty)$: they gave a direct proof of the sub-duality theorem just mentioned, and obtained a positive result for general exponents under additional side conditions of `weak boundedness' type (but rather more technical than the usual forms of such assumptions). See also \cite{HyMa:localTb,LaMa:localTb,LaVa:perfect} for some related work in the direction of singular integrals with respect to more general underlying measures.

In the paper at hand, we solve a buffered version of Hofmann's problem for all exponents $p,q\in(1,\infty)$; namely, we strengthen the last two inclusions in \eqref{eq:Hofmann} to
\begin{equation}\label{eq:Hofmann2}
  b^1_Q\in L^p,\quad b^2_Q\in L^q,\quad 1_{2Q}Tb^1_Q\in L^{q'},\quad 1_{2Q}T^*b^2_Q\in L^{p'}
\end{equation}
--- a condition also present in \cite{AR}, where further assumptions were imposed. In the sub-duality case $1/p+1/q\leq 1$, the buffered assumption \eqref{eq:Hofmann2} already follows from the seemingly weaker \eqref{eq:Hofmann} (cf. Remark \ref{rem:bufferRedundant}). In a recent work of Martikainen, Mourgoglou and Tolsa \cite{MMT}, the buffering is shown to be redundant in a slightly bigger range of exponents, but their main result for local $Tb$ theorems \cite[Theorem 4.6]{MMT}, while also allowing more general underlying measures, is restricted to antisymmetric kernels, $K(x,y)=-K(y,x)$. Our present contribution, while not strictly comparable to \cite{HyMa:localTb,LaMa:localTb,LaVa:perfect,MMT}, is the most general available form of local $Tb$ theorems for general (not necessarily antisymmetric) singular integrals with respect to the underlying Lebesgue measure.

We are going to view \eqref{eq:Hofmann2} as sufficient conditions for another natural set of assumptions stated in terms of the maximal truncated singular integral
\begin{equation*}
  T_{\#}f(x):=\sup_{\eps>0}\abs{T_\eps f(x)},\qquad T_\eps f(x):=\int_{\abs{x-y}>\eps}K(x,y)f(y)\ud y.
\end{equation*}
We make the assumption that for some $u\in(1,\infty)$ and $v\in[0,\infty)$ (sic!), we have
\begin{equation}\label{eq:new}
  b^1_Q,b_Q^2\in L^u,\quad 1_Q T_{\#}b^1_Q,\ 1_Q(T^*)_{\#}b^2_Q \in L^v,
\end{equation}
with appropriate uniform scaling. As we will prove, \eqref{eq:new} for $u=v<\min\{p,q,p',q'\}$ is a consequence of \eqref{eq:Hofmann}. But \eqref{eq:new} seems more natural in the sense that (unlike \eqref{eq:Hofmann} in the super-duality case $1/p+1/q>1$) it is obviously necessary for the $L^2$-boundedness of $T$, which implies the $L^u$-boundedness of $T_{\#}$ by classical theory. And \eqref{eq:new}, even with just a uniform $L^0$-condition for the images, is also a sufficient condition for the $L^2$-boundedness, as a proper $Tb$ condition should.

Our method of proving this result is new in the context of Hofmann's problem, although borrowed from other developments in the $Tb$ circle of ideas, in particular, the approach to Vitushkin's conjecture by Nazarov, Treil and Volberg \cite{NTV:Vitushkin}. Namely, we show that it is possible to perturb the rough test functions $b_Q\in L^u$ so as to obtain better functions $\tilde{b}_Q\in L^\infty$, which are still well-behaved under a \emph{suppressed} singular integral
\begin{equation*}
  T_\Phi f(x):=\int_{\R^d}K_\Phi(x,y)f(y)\ud y,\qquad K_\Phi(x,y):=\frac{\abs{x-y}^{2m}K(x,y)}{\abs{x-y}^{2m}+\Phi(x)^m\Phi(y)^m},
\end{equation*}
for a suitably chosen nonnegative Lipschitz function $\Phi$. We can then run a local $Tb$ argument for the suppressed operator $T_\Phi$ and the bounded test functions $\tilde{b}_Q$. Once the boundedness of $T_\Phi$ has been established, this can be used to construct yet another set of bounded test functions, but now for the original operator $T$. Another local $Tb$ argument with bounded test functions then allows to deduce the boundedness of $T$ itself.

In the following section, we give a detailed statement of the main theorems and a technical outline of the entire argument, where the main auxiliary propositions are stated without proof. The proofs of these intermediate results are then provided in the subsequent sections.

\subsection*{Acknowledgements}
This research was started during T.H.'s visit to the Kent State University in October 2011.
Over different periods of the duration of this work, T.H. was supported by the European Union through the ERC Starting Grant ``Analytic--probabilistic methods for borderline singular integrals'' (Grant Agreement No.~278558), as well as by the Academy of Finland through project Nos.~130166, 133264 and 314829 and through the Finnish Centre of Excellence in Analysis and Dynamics Research (project Nos.~271983 and  307333). 

{The anonymous referee is gratefully acknowledged for his or her thorough reading and friendly comments that eliminated an embarrassing number of careless details from the manuscript.}

\section{Technical outline}

Let $T$ be a linear operator given by
\begin{equation}\label{eq:defT}
  Tf(x)=\int_{\R^d} K(x,y)f(y)\ud y,
\end{equation}
where $K$ is a Calder\'on--Zygmund standard kernel:
\begin{equation}\label{eq:CZK}
  \abs{x-y}^d\abs{K(x,y)}+
  \abs{x-y}^{d+\alpha}\frac{\abs{K(x,y)-K(x',y)}}{\abs{x-x'}^\alpha}+
  \abs{x-y}^{d+\alpha}\frac{\abs{K(x,y)-K(x,y')}}{\abs{y-y'}^\alpha}
  \lesssim 1
\end{equation}
for all $x,x',y,y'$ with $\abs{x-x'}+\abs{y-y'}<\tfrac12\abs{x-y}$ and some fixed $\alpha\in(0,1]$. For convenience, we assume that $K$ is also bounded, \emph{qualitatively}, so that formulae like \eqref{eq:defT} are meaningful, but this will never be used in the quantitative estimates.

\begin{definition}[Accretive system]
Let $p\in[1,\infty]$, $u\in[0,\infty]$. A \emph{$(p,u)$-accretive system for an operator $T$} is a family of functions $b_Q$, indexed by all dyadic cubes $Q$, such that
\begin{equation}\label{eq:bSupport}
  \supp b_Q\subseteq Q,
\end{equation}
\begin{equation}\label{eq:bNondeg}
  \fint_Q b_Q \ud x = 1,
\end{equation}
\begin{equation}\label{eq:bSize}
  \Big(\fint_Q\abs{b_Q}^p\ud x\Big)^{1/p}\lesssim 1,
\end{equation}
\begin{equation}\label{eq:TbSize}
  \Big(\fint_Q\abs{Tb_Q}^{u}\ud x\Big)^{1/u}\lesssim 1,
\end{equation}
with the usual reformulation if $p$ or $u$ is $\infty$. If $u=0$, we interpret \eqref{eq:TbSize} as follows: For every $\eps>0$, there exists $\lambda$ such that for all $Q$ we have
\begin{equation}\label{eq:TbSize0}
  \frac{1}{\abs{Q}}\abs{Q\cap\{\abs{Tb_Q}>\lambda\}}<\eps.
\end{equation}
The family is called a \emph{buffered $(p,u)$-accretive system for $T$} if \eqref{eq:TbSize} is replaced by the stronger condition that
\begin{equation}\label{eq:TbSize2}
  \Big(\fint_{2Q}\abs{Tb_Q}^{u}\ud x\Big)^{1/u}\lesssim 1,
\end{equation}
with similar modifications for $u\in\{0,\infty\}$.
\end{definition}

Addressing a problem posed by Hofmann \cite[Section 3.3.1]{Hof:Escorial}, we prove the following:

\begin{theorem}\label{thm:Hofmann}
Let $p,q\in(1,\infty)$. Suppose that there is a buffered $(p,q')$-accretive system for $T$, and a buffered $(q,p')$-accretive system for $T^*$.
Then $\Norm{T}{L^2\to L^2}$ is bounded by a constant depending only on the implied constants in \eqref{eq:CZK}, \eqref{eq:bSize} and \eqref{eq:TbSize2}.
\end{theorem}

The first theorem of this flavour was proven for so-called perfect dyadic operators \cite{AHMTT}. For Calder\'on--Zygmund operators, prior to our work, it was known in the subduality case: $1/p+1/q\leq 1$ \cite{AR,AY}. For $1/p+1/q>1$, it had only been verified under additional technical assumptions~\cite{AR}.

\begin{remark}\label{rem:bufferRedundant}
In the subduality case, a $(p,q')$-accretive system is automatically buffered, i.e., \eqref{eq:TbSize2} already follows from the other conditions by the following Hardy inequality:
\begin{equation*}
   \int_{3Q\setminus Q}\Big(\int_Q\frac{\abs{f(y)}}{\abs{x-y}^d}\ud y\Big)^u\ud x
   \lesssim\Norm{1_Q f}{L^u}^u,\qquad u\in(1,\infty),
\end{equation*}
{whose dualized form is recorded in \cite[p.~307]{AR}; see also \cite[Lemma 2.4]{AR} for a more general statement (with proof) in a doubling metric space in place of $\R^d$.}

Our argument still requires the buffering for general exponents.
\end{remark}

%


We deduce Theorem~\ref{thm:Hofmann} as a corollary to a variant, natural in our opinion, where \eqref{eq:TbSize2} is replaced by an assumption on the maximal truncated operator
\begin{equation*}
  T_{\#}f(x):=\sup_{\varepsilon>0}\abs{T_\varepsilon f(x)},\qquad T_{\varepsilon}f(x):=\int_{\abs{x-y}>\varepsilon}K(x,y)f(y)\ud y.
\end{equation*}
Let us first observe that the above conditions on $T$ imply similar conditions on $T_{\#}$:

\begin{proposition}\label{prop:testTtestTmax}
Suppose that $b^1_Q$ is a buffered $(p,q')$ accretive system for $T$, and $b^2_Q$ is a buffered $(q,p')$ accretive system for $T^*$. Then $b^1_Q$ is also a  (buffered) $(p,u)$ accretive system for $T_{\#}$, and $b^2_Q$ is a (buffered) $(q,v)$ accretive system for $(T^*)_{\#}$, for any $u<\min\{p,q'\}$ and any $v<\min\{q,p'\}$.
\end{proposition}

Since $u<p$, the $(p,u)$ accretive system $b^1_Q$ is automatically buffered by Hardy's inequality; similarly for $b^2_Q$. Thanks to Proposition~\ref{prop:testTtestTmax}, Theorem~\ref{thm:Hofmann} is a consequence of our main result which reads as follows:

\begin{theorem}[Main theorem]\label{thm:main}
Let $T$ be a Calder\'on--Zygmund operator, and suppose that for some $p\in(1,\infty)$ and $u\in[0,\infty)$ there exist $(p,u)$-accretive systems 
$b^1_Q$ for $T_{\#}$ and $b^2_Q$ for $(T^*)_{\#}$. Then $\Norm{T}{L^2\to L^2}$ is bounded by a constant depending only on the implied constants in \eqref{eq:CZK}, \eqref{eq:bSize} and \eqref{eq:TbSize} with $T_{\#}$ and $b^1_Q$, or $(T^*)_{\#}$ and $b^2_Q$, in place of $T$ and $b_Q$.
\end{theorem}

\begin{remark}[Necessity of the conditions]
As in the usual $Tb$ theorems, the assumptions of Theorem~\ref{thm:main} are also necessary. Namely, if $T$ is $L^2$-bounded, it follows from standard theory that the maximal truncation $T_{\#}$ is $L^p$-bounded for all $p\in(1,\infty)$. Thus, for an $L^2$-bounded $T$, any function $b_Q$ with properties \eqref{eq:bSupport} and \eqref{eq:bSize} will also satisfy \eqref{eq:TbSize} for $T_{\#}$ in place of $T$.
\end{remark}

\begin{lemma}\label{lem:reduction}
It suffices to prove Theorem~\ref{thm:main} for $u=p>1$, as long as the implied constants in \eqref{eq:TbSize} for $T_{\#}$ and $(T^*)_{\#}$ contribute linearly to the bound for $\Norm{T}{2\to 2}$.
\end{lemma}

\begin{proof}
Indeed, suppose that the theorem is already proven in this case, and consider the weakest assumptions with $(p,0)$-accretive systems for $T_{\#}$ and $(T^*)_{\#}$.
Fix some $u\in(1,p)$. Then we have
\begin{equation*}
\begin{split}
  \fint_Q (T_{\#}b_Q^1)^u\ud x
  &=\frac{1}{\abs{Q}}\Big(\int_{Q\cap\{T_{\#}b_Q^1\leq\lambda\}}+\int_{Q\cap\{T_{\#}b^1_Q>\lambda\}}\Big)(T_{\#}b_Q^1)^u\ud x \\
  &\leq\lambda^u+\Big(\fint_Q (T_{\#}b_Q^1)^p\ud x\Big)^{u/p}\Big(\frac{1}{\abs{Q}}\abs{Q\cap\{T_{\#}b^1_Q>\lambda\}}\Big)^{1-u/p} \\
  &\leq\lambda^u+\Norm{T_{\#}}{p\to p}^u\Big(\fint_Q \abs{b_Q^1}^p\ud x\Big)^{u/p}\eps^{1-u/p},
\end{split}
\end{equation*}
if $\lambda=\lambda(\eps)$ is as in \eqref{eq:TbSize0}, and hence
\begin{equation*}
\begin{split}
  \Big(\fint_Q (T_{\#}b_Q^1)^u\ud x\Big)^{1/u}
  &\leq\lambda+C\Norm{T_{\#}}{p\to p}\cdot\eps^{1/u-1/p} \\
  &\leq \lambda+C(1+\Norm{T}{2\to 2})\cdot\eps^{1/u-1/p}
  \leq C_\delta+\Norm{T}{2\to 2}\cdot\delta,
\end{split}
\end{equation*}
where $\delta=C\eps^{1/u-1/p}$ can be chosen as small as we like, picking $\lambda=\lambda(\eps)$, and therefore $C_\delta=\lambda+C\eps^{1/u-1/p}$ large enough. So in fact $b_Q^1$ is also a $(p,u)$-accretive system (and, a fortiori, a $(u,u)$-accretive system) for $T_{\#}$, with the implied constant in the last estimate specified above. A similar argument of course works for $(T^*)_{\#}$ and $b^2_Q$. As these bounds contribute linearly to the bound for $\Norm{T}{2\to 2}$, we deduce from the assumed case of Theorem~\ref{thm:main} that
\begin{equation*}
  \Norm{T}{2\to 2}\leq C_\delta+c\Norm{T}{2\to 2}\cdot\delta,
\end{equation*}
and it suffices to fix a small enough $\delta>0$ in order to absorb the last term and complete the proof of Theorem~\ref{thm:main} in general.
\end{proof}

Let us then discuss the proof of Theorem~\ref{thm:main} for $1<u=p<\infty$; we will not pay explicit attention to the linear dependence required by Lemma~\ref{lem:reduction}, as it will be obvious by inspection. The proof consists of a reduction to the easier case of Hofmann's conjecture with the help of so-called suppressed operators. For any nonnegative function $\Phi$ with Lipschitz constant $1$, we define
\begin{equation*}
  T_\Phi f(x):=\int K_\Phi(x,y)f(y)\ud y,\qquad K_\Phi(x,y):=\frac{\abs{x-y}^{2m}K(x,y)}{\abs{x-y}^{2m}+\Phi(x)^m\Phi(y)^m},
\end{equation*}
where $m\geq d/2$ is fixed.
Note that $T_\Phi f=Tf$ if $\supp f\subseteq\{\Phi=0\}$.

Given the assumptions on $T_{\#}$, our goal is to construct a better behaved accretive system for the suppressed operator $T_\Phi$. To achieve this, we need to relax the notion of an accretive system a little, so as not to demand the supply of test functions for every cube, but only an appropriate subcollection of them.

\begin{definition}[Accretive system on a sparse family] 
Let $Q_0$ be a cube. A \emph{sparse family} of dyadic subcubes of $Q_0$ is a collection $\mathscr{Q}$, containing $Q_0$, such that for some $\tau>0$ and for all $Q\in\mathscr{Q}$ we have
\begin{equation*}
   \Babs{\bigcup_{\substack{\tilde{Q}\in\mathscr{Q}\\ \tilde{Q}\subsetneq Q}}\tilde{Q}}\leq(1-\tau)\abs{Q}.
\end{equation*}

An \emph{$(\infty,u)$-accretive system for $T$ on sparse subcubes of $Q_0$} is a family of functions $b_Q$, indexed by a sparse family $\mathscr{Q}$ of dyadic subcubes of $Q_0$, with the properties \eqref{eq:bSupport} {and} the following strengthening of \eqref{eq:bNondeg}, \eqref{eq:bSize} and \eqref{eq:TbSize}:
\begin{equation*}
\begin{split}
  \Babs{\fint_{Q'}b_Q\ud x} &\gtrsim 1, \\
  \Norm{b_Q}{\infty} &\lesssim 1, \\
  \Big(\fint_{Q'}\abs{Tb_Q}^u\ud x\Big)^{1/u} &\lesssim 1,
\end{split}
\end{equation*}
whenever $Q'\subseteq Q\in\mathscr{Q}$ is a dyadic subcube, which is not contained in any smaller $\tilde{Q}\subsetneq Q$ with $\tilde{Q}\in\mathscr{Q}$.

Such a system is said to be \emph{strong} if, in addition, we have
\begin{equation}\label{eq:strongAccr}
  \Big(\fint_{Q'}\abs{T(1_{Q'}b_Q)}^u\ud x\Big)^{1/u}\lesssim 1
\end{equation}
under the same conditions on the cubes $Q'\subseteq Q$.
\end{definition}

It is important for us that the additional condition~\eqref{eq:strongAccr} is automatically satisfied in the situations of our interest:

\begin{lemma}\label{lem:strongAccr}
Let $u\in(1,\infty)$ and suppose that there exists a $(u',1)$ accretive system for $T^*$. Then any $(\infty,u)$ accretive system for $T$ on sparse subcubes of any $Q_0$ is automatically strong.
\end{lemma}

The following proposition provides the key existence result for $(\infty,u)$ accretive systems:

\begin{proposition}\label{prop:testTPhi}
Suppose that there is a $(p,p)$ accretive system for $T_{\#}$ with $p>1$. Fix a small $\varrho\in(0,1)$ and any $u\in(1,\infty)$.
 Then for any cube $Q_0$, there exists a nonnegative function $\Phi$ with Lipschitz constant $1$ such that
\begin{equation*}
  \abs{\{\Phi>0\}}\leq\varrho\abs{Q_0},
\end{equation*}
and there exists an $(\infty,u)$ accretive system for $T_\Phi$ on sparse subcubes of $Q_0$.

Starting from two $(p,p)$ accretive systems, $b^1_Q$ for $T_{\#}$ and $b^2_Q$ for $(T^*)_{\#}$, this gives us two Lipschitz functions $\Phi_1$ and $\Phi_2$, and two $(\infty,u)$ accretive systems on sparse subcubes 
--- $\tilde{b}^1_Q$ for $T_{\Phi_1}$ and $\tilde{b}^2_Q$ for $T^*_{\Phi_2}$ ---, but we may arrange the construction so that $\Phi_1=\Phi_2=:\Phi$.

In this case, the two new accretive systems are strong.
\end{proposition}

Given an accretive system for $T$ on all dyadic subcubes of $Q_0$, it follows from a standard stopping time argument that we can extract a subsystem, which is an accretive system for $T$ on a sparse family of subcubes of $Q_0$. This is typically one of the first steps in the proof of a local $Tb$ theorem; see \cite{NTV:Duke}, for instance. For us, it will be important that it is actually enough to only have an accretive system for the sparse subcubes from the beginning:

\begin{proposition}[Baby $Tb$ theorem]\label{prop:babyTb}
Let $T$ be an operator with Calder\'on--Zygmund kernel, let $Q_0$ be a cube, { let $t\in(1,\infty)$}, and suppose that there are strong $(\infty,t)$ accretive systems $b^1_Q$ for $T$ and $b^2_Q$ for $T^*$, on sparse subcubes $\mathscr{Q}_1$ and $\mathscr{Q}_2$ of $Q_0$, respectively.
Then
\begin{equation*}
  \abs{\pair{Tf}{g}}\lesssim\Norm{f}{s'}\Norm{g}{s'}\abs{Q_0}^{1-2/s'}\qquad s'\in(\max\{t',2\},\infty],
\end{equation*}
for all $f,g\in L^{s'}(Q_0)$ ($L^{s'}$ functions supported on $Q_0$).
\end{proposition}

Note that $(\infty,t)$ accretive systems are in particular $(r,r')$ accretive systems for $r=\max\{4,t'\}<\infty$, $r'=\min\{\tfrac43,t\}\leq t$. Hence, in principle, we are in the setting of Hofmann's conjecture in the known case that $1/p+1/q=2/r\leq 1/2<1$. But contrary to the usual, we only assume the presence of (strong) accretive systems on sparse subcubes of $Q_0$.

Assuming all the auxiliary results formulated above, we can now give:

\begin{proof}[Proof of Theorem~\ref{thm:main}]
By the reduction given by Lemma~\ref{lem:reduction}, we may assume that for some $p\in(1,\infty)$, there are two $(p,p)$ accretive systems of functions, $b^1_Q$ for $T_{\#}$ and $b^2_Q$ for $(T^*)_{\#}$. Without loss of generality, we may assume that $p\in(1,2)$.

Fix a cube $Q_0$. Then, by Proposition~\ref{prop:testTPhi}, there exists a nonnegative function $\Phi$ with Lipschitz constant $1$ such that
\begin{equation}\label{eq:PhiOftenZero}
  \abs{\{\Phi>0\}}\leq\varrho\abs{Q_0},
\end{equation}
for some fixed $\varrho\in(0,1)$ (independent of $Q_0$), and there exist strong $(\infty,p)$ accretive systems $\tilde{b}^1_Q$ for $T_\Phi$ and $\tilde{b}^2_Q$ for $T^*_\Phi$ on sparse subcubes of $Q_0$.

Thus Proposition~\ref{prop:babyTb}, applied to $T_\Phi$ in place of $T$, implies that
\begin{equation}\label{eq:afterBabyTb1}
  \abs{\pair{T_\Phi f}{g}}\lesssim\Norm{f}{s'}\Norm{g}{s'}\abs{Q_0}^{1-2/s'}\qquad s'\in(p',\infty],
\end{equation}
for all $f,g\in L^{s'}(Q_0)$.

Let
\begin{equation*}
  b_{Q_0}:=\frac{\abs{Q_0}}{\abs{Q_0\cap\{\Phi=0\}}}1_{Q_0\cap\{\Phi=0\}}.
\end{equation*}
By \eqref{eq:PhiOftenZero}, we have $\abs{Q_0}/\abs{Q_0\cap\{\Phi=0\}}\lesssim 1$, and hence
\begin{equation*}
  \fint_{Q_0}b_{Q_0}\ud x=1,\qquad\Norm{b_{Q_0}}{\infty}\lesssim 1.
\end{equation*}
Since $\supp b_{Q_0}\subseteq\{\Phi=0\}$, we have $T_\Phi b_{Q_0}=Tb_{Q_0}$ and likewise $T_\Phi^* b_{Q_0}=T^*b_{Q_0}$. By an application of \eqref{eq:afterBabyTb1} to $f=b_{Q_0}$ and an arbitrary $g\in L^{s'}(Q_0)$ of norm $1$, we deduce that
\begin{equation*}
  \Big(\int_{Q_0}\abs{Tb_{Q_0}}^s\ud x\Big)^{1/s}\lesssim\abs{Q_0}^{1/s'}\cdot 1\cdot\abs{Q_0}^{1-2/s'}=\abs{Q_0}^{1/s}.
\end{equation*}
Similarly, applying \eqref{eq:afterBabyTb1} to $g=b_{Q_0}$ and an arbitrary $f\in L^{s'}(Q_0)$ of norm $1$, we obtain
\begin{equation*}
  \Big(\int_{Q_0}\abs{T^*b_{Q_0}}^s\ud x\Big)^{1/s}\lesssim\abs{Q_0}^{1/s}.
\end{equation*}

The above reasoning applies to any cube $Q$ in place of $Q_0$. Hence, for every $Q$, there exists a function $b_Q$ with
\begin{equation*}
  \supp b_Q\subseteq Q,\quad\fint_Q b_Q\ud x=1,\quad\Norm{b_Q}{\infty}\lesssim 1,\quad
  \fint_Q\abs{{Tb_Q}}^s\ud x+\fint_Q\abs{T^*b_Q}^s \lesssim 1.
\end{equation*}
In other words, there exists an $(\infty,s)$ accretive system for the original operator $T$ and its adjoint $T^*$ on all dyadic cubes. By a standard stopping time construction, for any $Q_0$, we can extract $(\infty,s)$ accretive systems for $T$ and $T^*$ on sparse subcubes of $Q_0$:
{Indeed, let $\mathscr Q_1$ consist of the maximal dyadic subcubes $Q\subseteq Q_0$ that satisfy at least one of the conditions
\begin{equation*}
  \Babs{\fint_Q b_{Q_0}\ud x}\leq\nu<1,\qquad
  \fint_Q\abs{Tb_{Q_0}}^s\ud x+\fint_Q\abs{T^*b_{Q_0}}^s>N\gg 1,
\end{equation*}
and $\mathscr Q_{1,0}:=\{Q\in\mathscr Q_1:\abs{\fint_Q b_{Q_0}\ud x}\leq\nu\}$. Then it is immediate that the opposite estimates hold for all dyadic $Q\subseteq Q_0$ that are not contained in any $\tilde Q\in\mathscr Q_1$, and
\begin{equation*}
\begin{split}
  \sum_{Q\in\mathscr Q_1\setminus\mathscr Q_{1,0}}\abs{Q}
  &\leq\frac{1}{N}\sum_{Q\in\mathscr Q_1\setminus\mathscr Q_{1,0}}\Big(\int_Q\abs{Tb_{Q_0}}^s\ud x+\int_Q\abs{T^*b_{Q_0}}^s\ud x\Big) \\
  &\leq\frac{1}{N}\Big(\int_{Q_0}\abs{Tb_{Q_0}}^s\ud x+\int_{Q_0}\abs{T^*b_{Q_0}}^s\ud x\Big) \leq\frac{C}{N}\abs{Q_0},
\end{split}
\end{equation*}
and
\begin{equation*}
\begin{split}
  \abs{Q_0}=\Babs{\int_{Q_0}b_{Q_0}\ud x}
  &\leq\sum_{Q\in\mathscr Q_{0,1}}\Babs{\int_Q b_{Q_0}\ud x}+\Babs{\int_{Q_0\setminus\bigcup\mathscr Q_{0,1}}b_{Q_0}\ud x} \\
  &\leq\sum_{Q\in\mathscr Q_{0,1}}\nu\abs{Q}+C\Big(\abs{Q_0}-\sum_{Q\in\mathscr Q_{0,1}}\abs{Q}\Big)
\end{split}
\end{equation*}
so that
\begin{equation*}
  \frac{1}{\abs{Q_0}}\sum_{Q\in\mathscr Q_{0}}\abs{Q}
  \leq\frac{1}{\abs{Q_0}}\Big(\sum_{Q\in\mathscr Q_{0,1}}\abs{Q}+\sum_{Q\in\mathscr Q_{0}\setminus\mathscr Q_{0,1}}\abs{Q}\Big)
  \leq\frac{C-1}{C-\nu}+\frac{C}{N}=:1-\tau<1
\end{equation*}
as soon as $\nu<1$ and $N\gg 1$ is large enough. We then iterate the same construction, starting from each $Q\in\mathscr Q_1$ in place of $Q_0$ and $b_Q$ in place of $b_{Q_0}$, to obtain a sparse family of cubes $\mathscr Q=\bigcup_{k=0}^\infty\mathscr Q_k$, where $\mathscr Q_0=\{Q_0\}$ and $\mathscr Q_{k+1}=\bigcup_{Q\in\mathscr Q_k}\mathscr Q_1(Q)$. By construction, the family $\{b_Q\}_{Q\in\mathscr Q}$ is an $(\infty,s)$-accretive system for both $T$ and $T^*$ on the sparse family $\mathscr Q$.}

Given that there also exist $(\infty,s)$ accretive systems, and thus a fortiori $(s',1)$ accretive systems, for the adjoints of both operators, we conclude from Lemma~\ref{lem:strongAccr} that both extracted systems are strong.

Another application of Proposition~\ref{prop:babyTb}, to the operator $T$ itself, shows that
\begin{equation*}
    \abs{\pair{Tf}{g}}\lesssim\Norm{f}{r'}\Norm{g}{r'}\abs{Q_0}^{1-2/r'}\qquad r'\in(s',\infty],
\end{equation*}
for all $f,g\in L^{r'}(Q_0)$, for any cube $Q_0$. We apply this to $f=1_{Q_0}$ and an arbitrary $g\in L^{r'}(Q_0)$ of norm $1$, and to $g=1_{Q_0}$ and an arbitrary $f\in L^{r'}(Q_0)$ of norm $1$, to deduce that
\begin{equation*}
  \Big(\int_{Q_0}\abs{T1_{Q_0}}^r\ud x\Big)^{1/r}\lesssim\abs{Q_0}^{1/r},\qquad
  \Big(\int_{Q_0}\abs{T^*1_{Q_0}}^r\ud x\Big)^{1/r}\lesssim\abs{Q_0}^{1/r}.
\end{equation*}
But this brings us to the setting of the well-known standard local $T1$ theorem {(see e.g.~\cite[Theorem 1.5 and Remark 1.7]{Hof:Escorial})}, which gives us the desired bound $\Norm{T}{L^2\to L^2}\lesssim 1$. This completes the proof.
\end{proof}

\section{Preparatory estimates; proofs of Proposition~\ref{prop:testTtestTmax} and Lemma~\ref{lem:strongAccr}}

We show how to obtain the testing conditions for the maximal truncated operator $T_{\#}$ from testing conditions for $T$, and provide some auxiliary results on the suppressed operators $T_\Phi$ for the subsequent sections.

\begin{lemma}\label{lem:TbSize}
{Let $p\in[1,\infty]$ and $u\in(1,\infty]$.}
If $b^1_Q$ is a buffered $(p,u)$ accretive system for $T$, then \eqref{eq:TbSize2} improves to the global estimate
\begin{equation*}
  \Norm{Tb^1_Q}{L^{u}}\lesssim\abs{Q}^{1/u}.
\end{equation*}
\end{lemma}

\begin{proof}
By \eqref{eq:TbSize2}, it only remains to estimate $1_{(2Q)^c}Tb^1_Q$. But
\begin{equation*}
\begin{split}
  \Norm{1_{(2Q)^c}Tb^1_Q}{L^{u}}
  &=\BNorm{x\mapsto \int_Q K(x,y)b^1_Q(y)\ud y}{L^{u}((2Q)^c)} \\
  &\leq\int_Q \Norm{x\mapsto  K(x,y)}{L^{u}((2Q)^c)}\abs{b^1_Q(y)}\ud y 
  \lesssim\int_Q\abs{Q}^{-1/u'}\abs{b^1_Q(y)}\ud y
  \lesssim\abs{Q}^{1/u},
\end{split}
\end{equation*}
by a straightforward estimate of the $L^{u}((2Q)^c)$ norm of $x\mapsto \abs{K(x,y)}\lesssim\abs{x-y}^{-d}$ for { fixed $y\in Q$ and variable $x\in(2Q)^c$.}
\end{proof}

We have the following version of Cotlar's lemma:

\begin{lemma}\label{lem:TmaxDom}
Suppose that there is a buffered $(q,v)$ accretive system for $T^*$. Then
\begin{equation*}
  T_{\#}f\lesssim M_{q'}(Tf)+M_{v'}f.
\end{equation*}
\end{lemma}

\begin{remark}
The classical Cotlar inequality states that $T_{\#}f\lesssim M(Tf)+Mf$ \emph{assuming that $T$ is already bounded on $L^2$}. Thus, trying to use this estimate in an attempt to prove the $L^2$-boundedness of $T$ would result in a circular argument. While we have assumed the $L^2$-boundedness of $T$ qualitatively, an inspection of the proof of Cotlar's inequality shows that $\Norm{T}{}$ enters quantitatively into the estimate; in fact, a more precise form would be
\begin{equation*}
  T_{\#}f\lesssim M(Tf)+(1+\Norm{T}{L^1\to L^{1,\infty}})Mf,
\end{equation*}
which is of no use to us. So the point of Lemma~\ref{lem:TmaxDom} is to establish a version of Cotlar's inequality without using the boundedness of $T$, but only the existence of an accretive system for $T^*$.
\end{remark}

\begin{proof}[Proof of Lemma~\ref{lem:TmaxDom}]
Fix $x_0$ and $\eps>0$. Let $Q$ be the unique dyadic cube containing $x_0$ and of diameter {in the range $(\frac18\eps,\frac14\eps]$}; thus $B(x_0,c_d\eps)\subset 2Q\subset B(x_0,\eps)$.
For all $x\in Q$, we have
\begin{equation*}
  T_{\eps}f(x_0)=T_{\eps}f(x_0)-T(f 1_{(2Q)^c})(x)+Tf(x)-T(f1_{2Q})(x),
\end{equation*}
where
\begin{equation*}
\begin{split}
  &\abs{T_{\eps}f(x_0)-T(f 1_{(2Q)^c})(x)}
  =\Babs{\int_{\abs{y-x_0}>\eps}K(x_0,y)f(y)\ud y-\int_{(2Q)^c}K(x,y)f(y)\ud y} \\
  &\leq\int_{(2Q)^c}\abs{K(x_0,y)-K(x,y)}\abs{f(y)}\ud y+\int_{\substack{\abs{y-x_0}\leq \eps\\ y\in (2Q)^c}}\abs{K(x_0,y)}\abs{f(y)}\ud y\\
  &\lesssim\int_{\abs{y-x_0}>c_d\eps}\frac{\eps^\alpha}{\abs{y-x_0}^{d+\alpha}}\abs{f(y)}\ud y
     +\int_{c_d\eps<\abs{y-x_0}<\eps}\frac{1}{\eps^d}\abs{f(y)}\ud y 
  \lesssim Mf(x_0).
\end{split}
\end{equation*}
Thus
\begin{equation*}
\begin{split}
  T_{\eps}f(x_0)
  &=\fint_{Q}b^2_{Q}\ud x\cdot T_{\eps}f(x_0) \\
  &=\mathcal{O}(Mf(x_0))+\fint_Q b^2_{Q}\cdot Tf\ud x
       -\frac{1}{\abs{Q}}\int T^*b^2_Q\cdot f1_{2Q}\ud x \\
  &\lesssim Mf(x_0)+\Big(\fint_Q \abs{b^2_{Q}}^q\Big)^{1/q}\cdot \Big(\fint_Q\abs{Tf}^{q'}\ud x\Big)^{1/q'} \\
   &\qquad    +\Big(\fint_{2Q}\abs{T^*b^2_Q}^v\ud x\Big)^{1/v}\Big(\fint_{2Q}\abs{f}^{v'}\ud x\Big)^{1/v'} \\
   &\lesssim Mf(x_0)+M_{q'}(Tf)(x_0)+M_{v'}f(x_0).\qedhere
\end{split}
\end{equation*}
\end{proof}

\begin{lemma}
Let $b^1_Q$ be a buffered $(p,q')$ accretive system for $T$, and $b^2_Q$ a buffered $(q,p')$ accretive system for $T^*$. Then
\begin{equation*}
 \abs{Q}^{-1/u}\Norm{1_{2Q}T_{\#}b^1_Q}{L^u}\lesssim
  \abs{Q}^{-1/s}\Norm{1_{2Q}T_{\#}b^1_Q}{L^{s,\infty}}
  \lesssim 1,\qquad u<s:=\min\{p,q'\}.
\end{equation*}
In particular, $b^1_Q$ is a buffered $(p,u)$ accretive system for $T_{\#}$ for all $u<\min\{p,q'\}$.
\end{lemma}

\begin{proof}
We apply Lemma~\ref{lem:TmaxDom}, the quasi-triangle inequality and the monotonicity in the exponent of the weak norms, the boundedness of the maximal operators $M_p:L^p\to L^{p,\infty}$ and $M_{q'}:L^{q'}\to L^{q',\infty}$, Lemma~\ref{lem:TbSize}, and finally the assumptions on $b^1_Q$ to deduce
\begin{equation*}
\begin{split}
  \abs{Q}^{-1/s}\Norm{1_{2Q}T_{\#}b^1_Q}{L^{s,\infty}}
  &\lesssim\abs{Q}^{-1/s} \Norm{1_{2Q}[M_{q'}(Tb^1_Q)+M_pb^1_Q]}{L^{s,\infty}} \\
  &\lesssim\abs{Q}^{-1/q'} \Norm{1_{2Q}M_{q'}(Tb^1_Q)}{L^{q',\infty}}+\abs{Q}^{-1/p}\Norm{1_{2Q}M_pb^1_Q}{L^{p,\infty}} \\
  &\lesssim\abs{Q}^{-1/q'} \Norm{Tb^1_Q}{L^{q'}}+\abs{Q}^{-1/p}\Norm{b^1_Q}{L^{p}} \lesssim 1.\qedhere
\end{split}
\end{equation*}
\end{proof}

Note that we have now completed the proof of Proposition~\ref{prop:testTtestTmax}, and we conclude the section with:

\begin{proof}[Proof of Lemma~\ref{lem:strongAccr}]
Let $b^1_Q$, $Q\in\mathscr{Q}$, be an $(\infty,u)$ accretive system for $T$ on sparse subcubes of $Q_0$, and let $b^2_Q$ be a $(u',1)$ accretive system for $T^*$ (on all cubes). We need to show that $b^1_Q$ is strong. To this end, fix some $Q'\subseteq Q\in\mathscr{Q}$, where $Q'$ is not contained in any smaller $Q''\in\mathscr{Q}$.
We start with
\begin{equation}\label{eq:WBPstep1}
\begin{split}
  \Norm{1_{Q'} T(1_{Q'}b^1_Q)}{u}
  &=\Norm{1_{Q'}Tb^1_Q-1_{Q'}T(1_{3Q'\setminus Q'}b^1_Q)-1_{Q'}T(1_{(3Q')^c}b^1_Q)}{u} \\
  &\lesssim\abs{Q'}^{1/u}+\Norm{1_{3Q'\setminus Q'}b^1_Q}{u}+\abs{Q'}^{1/u}\Norm{1_{Q'}T(1_{(3Q')^c}b^1_Q)}{\infty}
\end{split}
\end{equation}
where the first term was estimated directly from the definition of an $(\infty,u)$ accretive system, and {the} second term by Hardy's inequality. Moreover, the $L^\infty$ bound implies that
\begin{equation*}
  \Norm{1_{3Q'\setminus Q'}b^1_Q}{u}\lesssim\abs{Q'}^{1/u},
\end{equation*}
so it only remains to estimate the third term.

For this task, we call to our service the $(u',1)$ accretive system $b^2_Q$, which we have for every dyadic cube, so in particular for $Q'$. Let $x\in Q'$, and write
\begin{equation*}
\begin{split}
  T(1_{(3Q')^c}b_Q^1)(x)
  &=\fint_{Q'} [T(1_{(3Q')^c} b_Q^1)(x)-T(1_{(3Q')^c} b_Q^1)(y)]b^2_{Q'}(y)\ud y \\
  &\qquad +\frac{1}{\abs{Q'}}\pair{T(1_{(3Q')^c} b_Q^1)}{b^2_{Q'}} \\
  &=\fint_{Q'}\int_{(3Q')^c}[K(x,z)-K(y,z)b_Q^1(z)\ud z\ b^2_{Q'}(y)\ud y \\
  &\qquad +\frac{1}{\abs{Q'}}\Big(\pair{Tb^1_Q}{b^2_{Q'}}
    -\pair{T(1_{3Q'\setminus Q'}b_Q^1 )}{b^2_{Q'}}-\pair{1_{Q'} b^1_Q }{T^* b^2_{Q'}}\Big).
\end{split}
\end{equation*}
The double integral is estimated by
\begin{equation*}
  \fint_{Q'}\int_{(3Q')^c}\frac{\ell(Q')^\alpha}{\abs{x-z}^{d+\alpha}}\ud z\ \abs{b^2_{Q'}(y)}\ud y
  \lesssim\fint_Q\abs{b^2_{Q'}(y)}\ud y\lesssim 1,
\end{equation*}
and we also have the bounds
\begin{equation*}
\begin{split}
  \abs{\pair{Tb^1_Q}{b^2_{Q'} }} &\lesssim\Norm{1_{Q'} Tb^1_Q }{u}\Norm{b^2_{Q'}}{u'}\lesssim\abs{Q'}^{1/u}\abs{Q'}^{1/u'}=\abs{Q'}, \\
  \abs{\pair{T(1_{3Q'\setminus Q'}b^1_Q)}{b^2_{Q'} }} &\lesssim\Norm{1_{3Q'\setminus Q'}{ b}^1_Q }{u}\Norm{b^2_{Q'}}{u'}
  \lesssim\abs{Q'}\qquad\text{by Hardy's inequality, and}\\
  \abs{\pair{1_{Q'} b^1_Q}{T^* b^2_{Q'}}} &\lesssim
  \Norm{1_Q b^1_Q}{\infty}\Norm{T^*b_{Q'}^2}{1}\lesssim 1\cdot\abs{{Q'}}=\abs{{Q'}}.
\end{split}
\end{equation*}
by the fact that $b^2_Q$ is a $(u',1)$-accretive system for $T^*$.
Thus
\begin{equation*}
  \Norm{T(1_{(3Q')^c} b^1_Q )}{\infty} \lesssim 1,
\end{equation*}
which, substituted to \eqref{eq:WBPstep1}, completes the proof that $\Norm{1_{Q'}T(1_{Q'}b^1_Q)}{u}\lesssim\abs{Q'}^{1/u}$.
\end{proof}

\section{Modified test functions; proof of Proposition~\ref{prop:testTPhi}}\label{sec:modTestF}

{Throughout this Section \ref{sec:modTestF}, we work under the assumptions of Proposition~\ref{prop:testTPhi}. In particular, for some $p>1$, we assume that $b_Q^1$ is a $(p,p)$ accretive system for the maximal truncated operator $T_{\#}$ and, where relevant, $b_Q^2$ is a similar system for $(T^*)_{\#}$.} Given these rough test functions for $T_{\#}$, we want to construct better test functions for $T_\Phi$, where $\Phi$ is suitably chosen. We start with some preparations.

\subsection{Generalities on suppressed operators}
For completeness, we detail the following observation, which is a routine extension of a special case appearing in \cite[Sec.~I]{NTV:Vitushkin}.

\begin{lemma}\label{lem:suppCZK}
Let $K$ be a Calder\'on--Zygmund standard kernel, $\Phi$ be a non-negative function with Lipschitz constant $1$, and $m\geq d/2$. Then the suprressed kernel $K_\Phi$ is also a Calder\'on--Zygmund standard kernel with the same H\"older exponent $\alpha\in(0,1]$.
\end{lemma}

\begin{proof}
We need to verify \eqref{eq:CZK} for $K_\Phi$ in place of $K$. Let us denote
\begin{equation*}
  L_\Phi(x,y):=\frac{\abs{x-y}^{2m}}{\abs{x-y}^{2m}+\Phi(x)^m\Phi(y)^m}
\end{equation*}
so that $K_\Phi=L_\Phi K$. It is immediate that $0\leq L_\Phi\leq 1$, and hence $\abs{K_\Phi(x,y)}\leq\abs{K(x,y)}\lesssim\abs{x-y}^{-d}$. Concerning smoothness, we have
\begin{equation*}
  L_\Phi(x,y)-L_\Phi(x',y)
  =\frac{\abs{x-y}^{2m}\Phi(x')^m\Phi(y)^m-\abs{x'-y}^{2m}\Phi(x)^m\Phi(y)^m}{
    [\abs{x-y}^{2m}+\Phi(x)^m\Phi(y)^m][\abs{x'-y}^{2m}+\Phi(x')^m\Phi(y)^m]}=:\frac{A}{B},
\end{equation*}
where
\begin{equation*}
  A=(\abs{x-y}^{2m}-\abs{x'-y}^{2m})\Phi(x')^m\Phi(y)^m+\abs{x'-y}^{2m}(\Phi(x')^m-\Phi(x)^m)\Phi(y)^m=:A_1+A_2.
\end{equation*}
We need to estimate these for $\abs{x-x'}\leq\frac12\abs{x-y}$, in which case $\abs{x-y}\eqsim\abs{x'-y}$ and
\begin{equation*}
   \Babs{\frac{A_1}{B}}\lesssim\frac{\abs{x-x'}\abs{x-y}^{2m-1}\Phi(x')^m\Phi(y)^m}{\abs{x-y}^{2m}\times\Phi(x')^m\Phi(y)^m}
   =\frac{\abs{x-x'}}{\abs{x-y}}.
\end{equation*}

Concerning $A_2$, we first observe that
\begin{equation*}
  \abs{\Phi(x')^m-\Phi(x)^m}
  =\abs{\Phi(x')-\Phi(x)}\sum_{k=0}^{m-1}\Phi(x')^{m-1-k}\Phi(x)^k
  \leq\abs{x'-x}\sum_{k=0}^{m-1}\Phi(x')^{m-1-k}\Phi(x)^k
\end{equation*}
by the Lipschitz assumption on $\Phi$, and hence
\begin{equation*}
  \abs{A_2}\leq\abs{x'-y}^{2m}\abs{x'-x}\Big(\sum_{k=0}^{m-1}(\Phi(x')\Phi(y))^{m-1-k}(\Phi(x)\Phi(y))^k\Big)\Phi(y).
\end{equation*}

Another observation is that $\abs{x-y}^{2m}+\Phi(x)^m\Phi(y)^m\eqsim(\abs{x-y}^2+\Phi(x)\Phi(y))^m$ and, again by the Lipschitz assumption on $\Phi$, we have $\Phi(x)\geq\Phi(y)-\abs{x-y}$ and hence
\begin{equation*}
  \abs{x-y}^2+\Phi(x)\Phi(y)\geq\abs{x-y}^2+\Phi(y)^2-\abs{x-y}\Phi(y)
  =(\abs{x-y}-\Phi(y))^2+\abs{x-y}\Phi(y)\geq\abs{x-y}\Phi(y).
\end{equation*}
Thus, for any $k=0,1,\ldots,m-1$,
\begin{equation*}
\begin{split}
  B &\eqsim[\abs{x-y}^2+\Phi(x)\Phi(y)]^{1+k+(m-1-k)}[\abs{x'-y}^2+\Phi(x')\Phi(y)]^{(m-1-k)+(k+1)} \\
  &\geq\abs{x-y}\Phi(y)\times(\Phi(x)\Phi(y))^k\times\abs{x-y}^{2(m-1-k)}\times 
  (\Phi(x')\Phi(y))^{m-1-k}\times\abs{x'-y}^{2(k+1)} \\
  &\eqsim\Phi(y)\times(\Phi(x)\Phi(y))^k\times(\Phi(x')\Phi(y))^{m-1-k}\times\abs{x'-y}^{2m+1} \\
\end{split}
\end{equation*}
and hence $\abs{A_2/B}\lesssim\abs{x'-x}/\abs{x'-y}\eqsim\abs{x'-x}/\abs{x-y}$. Together with the earlier estimates, we have checked that
\begin{equation*}
  \abs{L_\Phi(x,y)-L_\Phi(x',y)}\lesssim\frac{\abs{x-x'}}{\abs{x-y}},\qquad\abs{x-x'}\leq\frac12\abs{x-y}.
\end{equation*}
Thus, for the same range of the variables $x,x',y$,
\begin{equation*}
\begin{split}
  \abs{K_\Phi(x,y)-K_\Phi(x',y)}
  &\leq\abs{[L_\Phi(x,y)-L_\Phi(x',y)]K(x,y)}+\abs{L_\Phi(x',y)[K(x,y)-K(x',y)]} \\
  &\lesssim\frac{\abs{x-x'}}{\abs{x-y}}\frac{1}{\abs{x-y}^d}+1\cdot\frac{\abs{x-x'}^\alpha}{\abs{x-y}^{d+\alpha}}
  \lesssim\frac{\abs{x-x'}^\alpha}{\abs{x-y}^{d+\alpha}},\qquad\alpha\in(0,1].
\end{split}
\end{equation*}
This proves the smoothness estimate in \eqref{eq:CZK} in the first variable, and the case of the second variable is entirely analogous, since $L_\Phi$ itself is symmetric, and the assumptions on $K$ are symmetric in $x$ and $y$.
\end{proof}

\begin{lemma}\label{lem:TPhiF}
For any nonnegative function $\Phi$ with Lipschitz constant $1$, we have
\begin{equation*}
  \abs{T_{\Phi}f}\lesssim T_{\#}f+Mf.
\end{equation*}
\end{lemma}

\begin{proof}
\begin{equation*}
\begin{split}
  T_{\Phi}f(x)
  =\int K_{\Phi}(x,y)f(y)\ud y
  &=\int_{\abs{x-y}\leq \frac12\Phi(x)} K_{\Phi}(x,y)f(y)\ud y \\
  &\qquad +\int_{\abs{x-y}>\frac12\Phi(x)}(K_{\Phi}(x,y)-K(x,y))f(y)\ud y+T_{\frac12\Phi(x)}f(x)
\end{split}
\end{equation*}
For $\abs{x-y}\leq\tfrac12\Phi(x)$, we have $\Phi(y)\geq\Phi(x)-\abs{x-y}\geq\frac12\Phi(x)$, hence
\begin{equation*}
  \abs{K_\Phi(x,y)}\leq\abs{K(x,y)}\frac{\abs{x-y}^{2m}}{(\frac12\Phi(x)^2)^m}\lesssim\frac{\abs{x-y}^{2m-d}}{\Phi(x)^{2m}},
\end{equation*}
and thus
\begin{equation*}
\begin{split}
  \Babs{\int_{\abs{x-y}\leq \frac12\Phi(x)} K_{\Phi}(x,y)f(y)\ud y}
  &\lesssim\sum_{j=1}^{\infty}\int_{2^{-j-1}\Phi(x)<\abs{x-y}\leq 2^{-j}\Phi(x)} \frac{(2^{-j}\Phi(x))^{2m-d}}{\Phi(x)^{2m}}\abs{f(y)}\ud y \\
  &\lesssim\sum_{j=1}^{\infty}2^{-2jm}\fint_{\abs{x-y}\leq 2^{-j}\Phi(x)} \abs{f(y)}\ud y 
  \lesssim Mf(x).
\end{split}
\end{equation*}
And for $\abs{x-y}>\tfrac12\Phi(x)$, we have $\Phi(y)\leq \Phi(x)+\abs{x-y}<3\abs{x-y}$, so that
\begin{equation*}
  \abs{K_{\Phi}(x,y)-K(x,y)}
  =\abs{K(x,y)}\frac{(\Phi(x)\Phi(y))^m}{\abs{x-y}^{2m}+(\Phi(x)\Phi(y))^m}
  \lesssim\frac{1}{\abs{x-y}^d}\frac{\Phi(x)^m\abs{x-y}^m}{\abs{x-y}^{2m}},
\end{equation*}
and hence
\begin{equation*}
\begin{split}
  &\Babs{\int_{\abs{x-y}>\frac12\Phi(x)}(K_{\Phi}(x,y)-K(x,y))f(y)\ud y} \\
  &\lesssim\sum_{j=0}^{\infty}\int_{2^{j-1}\Phi(x)<\abs{x-y}\leq 2^j\Phi(x)}\frac{\Phi(x)^m}{(2^j\Phi(x))^{d+m}}\abs{f(y)}\ud y \\
  &\lesssim\sum_{j=0}^{\infty}2^{-jm}\fint_{\abs{x-y}\leq 2^j\Phi(x)}\abs{f(y)}\ud y \lesssim Mf(x).
\end{split}
\end{equation*}
Finally, it is clear that $\abs{T_{\frac12\Phi(x)}f(x)}\leq T_{\#}f(x)$.
\end{proof}

\begin{lemma}\label{lem:Tsharp2TPhi}
Let $b^1_Q$ be a (buffered) $(p,u)$ accretive system for $T_{\#}$, where $p>1$. Then it is also a (buffered) $(p,s)$ accretive system for any $T_\Phi$, where $s=\min(p,u)$.
\end{lemma}

\begin{proof}
Let $\alpha\in\{1,2\}$ according to whether the system is buffered $(\alpha=2)$ or not $(\alpha=1)$. By Lemma~\ref{lem:TPhiF} and the boundedness of $M$ on $L^p$, we have
\begin{equation*}
\begin{split}
  \abs{Q}^{-1/s}\Norm{1_{\alpha Q} T_\Phi b^1_Q}{L^{s}}
  &\lesssim\abs{Q}^{-1/s}\Norm{1_{\alpha Q} [T_{\#}b^1_Q+Mb^1_Q]}{L^{s}} \\
  &\leq\abs{Q}^{-1/u}\Norm{1_{\alpha Q} T_{\#}b^1_Q}{L^u}+\abs{Q}^{-1/p}\Norm{1_{\alpha Q}Mb^1_Q}{L^p}
  \lesssim 1.\qedhere
\end{split}
\end{equation*}
\end{proof}

\subsection{First step of the modification and key estimates}
We turn to the actual construction of the modified test functions $\tilde{b}^1_{Q}$.

Consider a fixed cube $Q_0$, and abbreviate $b:=b_{Q_0}^1$. Let $\mathscr{B}_1=\mathscr{B}_1(Q_0)$ be the maximal dyadic subcubes $Q\subseteq Q_0$ with $\fint_Q(Mb+T_{\#}b)^p\gg 1$, and $\Omega:=\bigcup_{Q\in\mathscr{B}_1}Q$. Let
\begin{equation}\label{eq:goodPart}
  \tilde{b}:=\tilde{b}_{Q_0}^1:=1_{\Omega^c}b+\sum_{Q\in\mathscr{B}_1}\ave{b}_Q 1_Q
  =b-\sum_{Q\in\mathscr{B}_1}(b-\ave{b}_Q)1_Q
  =:b-\sum_{Q\in\mathscr{B}_1}d_Q
\end{equation}
be (essentially) the good part of the Calder\'on--Zygmund decomposition of $b$.

\begin{lemma}\label{lem:TPhiDiff}
If $\Phi$ is a $1$-Lipschitz function with
\begin{equation}\label{eq:PhiLower}
  \Phi(x)\geq\sup_{Q\in\mathscr{B}_1}\dist(x,(3Q)^c), 
\end{equation}
then
\begin{equation*}
\begin{split}
  \abs{T_\Phi(b-\tilde{b})(x)}
  \leq \sum_{Q\in\mathscr{B}_1}\abs{T_\Phi d_Q(x)}
  &\lesssim\sum_{Q\in\mathscr{B}_1}\Big(\frac{\ell(Q)}{\ell(Q)+\abs{x-c_Q}}\Big)^{d+\alpha} \\
  &=:\sum_{Q\in\mathscr{B}_1}\phi_Q(x)=:e^1_{Q_0}(x),
\end{split}
\end{equation*}
where for all $u\in[1,\infty)$,
\begin{equation*}
  \Norm{e^1_{Q_0}}{L^u}\lesssim\abs{Q_0}^{1/u}.
\end{equation*}
\end{lemma}

\begin{proof}
{Recall that $K_\Phi$ is a Calder\'on--Zygmund standard kernel by Lemma \ref{lem:suppCZK}. Thus,} for $x\in(2Q)^c$,
\begin{equation*}
  \abs{T_\Phi d_Q(x)}
  =\Babs{\int_Q [K_\Phi(x,y)-K_\Phi(x,c_Q)]d_Q(y)\ud y} 
  \lesssim\frac{\ell(Q)^\alpha}{\abs{x-c_Q}^{d+\alpha}} \Norm{d_Q}{1}
  \lesssim\Big(\frac{\ell(Q)}{\abs{x-c_Q}}\Big)^{d+\alpha}.
\end{equation*}
For $x\in 2Q,y\in Q$, we have $\dist(x,(3Q)^c),\dist(y,(3Q)^c)\gtrsim \ell(Q)$, hence $\Phi(x)\Phi(y)\gtrsim\ell(Q)^2$, and therefore
\begin{equation}\label{eq:KPhiBound}
  \abs{K_\Phi(x,y)}\lesssim\frac{\abs{x-y}^{2m-d}}{\ell(Q)^{2m}}\lesssim\frac{1}{\ell(Q)^d}
\end{equation}
provided that $m\geq d/2$. Thus
\begin{equation*}
  \abs{T_\Phi d_Q(x)}=\Babs{\int_Q K_\Phi(x,y)d_Q(y)\ud y}\lesssim\frac{1}{\abs{Q}}\int_Q\abs{d_Q(y)}\ud y\lesssim 1.
\end{equation*}

Altogether we have
\begin{equation*}
   \abs{T_\Phi d_Q(x)} {\lesssim} \Big(\frac{\ell(Q)}{\ell(Q)+\abs{x-c_Q}}\Big)^{d+\alpha}=:\phi_Q(x).
\end{equation*}
By duality, for a suitable $g\geq 0$ with $\Norm{g}{u'}=1$,
\begin{equation*}
  \BNorm{\sum_{Q\in\mathscr{B}_1}\phi_Q}{u}
  =\int g\sum_{Q\in\mathscr{B}_1}\phi_Q
  \lesssim\sum_{Q\in\mathscr{B}_1}\abs{Q}\inf_Q Mg
  \leq\int Mg\sum_{Q\in\mathscr{B}_1}1_Q
  \leq\Norm{Mg}{u'}\BNorm{\sum_{Q\in\mathscr{B}_1}1_Q}{u},
\end{equation*}
where $\Norm{Mg}{u'}\lesssim\Norm{g}{u'}=1$ by the maximal inequality, and
\begin{equation*}
  \BNorm{\sum_{Q\in\mathscr{B}_1}1_Q}{u}=\Big(\sum_{Q\in\mathscr{B}_1}\abs{Q}\Big)^{1/u}\leq\abs{Q_0}^{1/u}
\end{equation*}
by the disjointness of the cubes $Q\in\mathscr{B}_1$.
\end{proof}

\begin{lemma}\label{lem:TPhiOldB}
If $\Phi$ is a $1$-Lipschitz function with
\begin{equation}\label{eq:PhiLower}
  \Phi(x)\geq\sup_{Q\in\mathscr{B}_1}\dist(x,(3Q)^c), 
\end{equation}
then
\begin{equation*}
   1_{Q_0}\abs{T_\Phi b}\lesssim 1
\end{equation*}
\end{lemma}

\begin{proof}
We have
\begin{equation*}
  1_{Q_0}\abs{T_\Phi b}
  =1_{Q_0\setminus{\Omega}}\abs{T_\Phi b}+\sum_{Q\in\mathscr{B}_1}1_Q \abs{T_\Phi b},
\end{equation*}
where
\begin{equation*}
  1_{Q_0\setminus{\Omega}}\abs{T_\Phi b}
  \lesssim 1_{Q_0\setminus{\Omega}}(T_{\#}b+Mb)\lesssim 1.
\end{equation*}

For any $x,y\in Q\in\mathscr{B}_1$, we find that
\begin{equation*}
\begin{split}
  \abs{T_\Phi b(x)-T_\Phi b(y)}
  &=\Babs{\int [K_\Phi(x,z)-K_\Phi(y,z)]b(z)\ud z} \\
  &\lesssim\int_{2Q}\frac{1}{\ell(Q)^d}\abs{b(z)}\ud z+\int_{(2Q)^c}\frac{\ell(Q)^\alpha}{\abs{z-x}^{d+\alpha}}\abs{b(z)}\ud z 
  \lesssim\inf_Q Mb\leq\fint_Q Mb\lesssim 1
\end{split}
\end{equation*}
by the maximality of the cubes $Q$. Thus
\begin{equation*}
  \sup_Q\abs{T_\Phi b}
  \lesssim 1+\inf_Q\abs{T_\Phi b}
  \lesssim 1+\inf_Q(T_{\#}b+Mb)
  \leq 1+\fint_Q(T_{\#}b+Mb)\lesssim 1,
\end{equation*}
again by the maximality of the cubes $Q$.
\end{proof}

We want to interpret the new function $\tilde{b}^1_{Q_0}$, and similarly constructed functions for subcubes of $Q_0$, as test functions for the operator $T_\Phi$. Note that the choice of $\Phi$ will be fixed only after a stopping time construction, by which we construct the remaining functions $\tilde{b}^1_Q$. Before we fix this choice, it is important that all the estimates are valid for any $\Phi$ with the property \eqref{eq:PhiLower}.

For any such $\Phi$, we have by Lemmas~\ref{lem:TPhiDiff} and \ref{lem:TPhiOldB} that
\begin{equation}\label{eq:TPhiTildeB}
  1_{Q_0}\abs{T_\Phi\tilde{b}^1_{Q_0}}
  \leq 1_{Q_0}\abs{T_\Phi b^1_{Q_0}}+1_{Q_0}\abs{T_\Phi(\tilde{b}^1_{Q_0}-b^1_{Q_0})} 
  \lesssim 1+e^1_{Q_0}.
\end{equation}

\subsection{Different stopping conditions}
We refer to the maximal dyadic subcubes $Q\subseteq Q_0$ with $\fint_Q(Mb^1_{Q_0}+T_{\#}b^1_{Q_0})^p\geq C/\delta$ as its \textbf{primary stopping cubes}. They satisfy
\begin{equation*}
  \sum_{Q\in\mathscr{B}_1(Q_0)}\abs{Q}
  =\abs{\{M^d(Mb^1_{Q_0}+T_{\#}b^1_{Q_0})^p>C/\delta\}}
  \leq\frac{1}{C/\delta}\int(Mb^1_{Q_0}+T_{\#}b^1_{Q_0})^p\leq\delta\abs{Q_0},
\end{equation*}
{where in the last step we use the standing assumption of Proposition \ref{prop:testTPhi} (which is in force throughout this Section \ref{sec:modTestF}) that $b_Q^1$ is a $(p,p)$ accretive system for the maximal truncated operator $T_{\#}$.}
The function $e^1_{Q_0}$ depends on these cubes, and thus on $\delta$; however, the bound $\Norm{e^1_{Q_0}}{u}\lesssim\abs{Q}^{1/u}$, 
{guaranteed by Lemma \ref{lem:TPhiDiff}, is independent of these parameters.}

The \textbf{secondary stopping cubes} of $Q_0$ are defined as the maximal dyadic subcubes $Q\subseteq Q_0$ that satisfy either of the following conditions: For a given $u\in(1,\infty)$, which we take at least as large as $\max\{p,p'\}$,
\begin{equation}\label{eq:stopTb}
  \fint_Q(e^1_{Q_0})^u> \frac{1}{\eps},
\end{equation}
or
\begin{equation}\label{eq:stopDegen}
  \Babs{\fint_Q \tilde{b}^1_{Q_0}}\leq \eta.
\end{equation}
The measure of the cubes in \eqref{eq:stopTb} is at most
\begin{equation*}
  \abs{\{M^d(1_{Q_0}(e^1_{Q_0})^u)>1/\eps\}}
  \leq\eps\int_{Q_0} (e^1_{Q_0})^u\lesssim\eps\abs{Q_0}.
\end{equation*}
For the cubes in \eqref{eq:stopDegen}, we compute
\begin{equation*}
\begin{split}
  \abs{Q_0} &=\Babs{\int b_{Q_0}^1}=\Babs{\int \tilde{b}_{Q_0}^1} \\
  &\leq\sum_Q\Babs{\int \tilde{b}_{Q_0}^1}+\int_{Q_0\setminus\bigcup Q}\abs{\tilde{b}_{Q_0}^1} \\
  &\leq\sum_Q\eta\abs{Q}+\abs{Q_0\setminus\bigcup Q}^{1/p'}\Norm{\tilde{b}_{Q_0}^1}{p} \\
  &\leq\eta\abs{Q_0}+C\Big(\abs{Q_0}-\sum_Q\abs{Q}\Big)^{1/p'}\abs{Q_0}^{1/p},
\end{split}
\end{equation*}
by using $\Norm{\tilde{b}_{Q_0}^1}{p}\lesssim\Norm{b^1_{Q_0}}{p}\lesssim\abs{Q_0}^{1/p}$ in the last step. From here one can solve
\begin{equation*}
  \sum\abs{Q}\leq\Big\{1-\Big(\frac{(1-\eta)}{C}\Big)^{p'}\Big\}\abs{Q_0}.
\end{equation*}
Altogether, taking $\eta<1$ and $\eps$ sufficiently small, the measure of the secondary stopping cubes is at most a fraction $(1-\tau)<1$ of $\abs{Q_0}$.

\subsection{Iteration of the stopping conditions}
We iterate the following algorithm, starting from an arbitrary but fixed dyadic cube $Q_0$.
\begin{itemize}
  \item We choose the primary stopping cubes $\mathscr{B}_1=\mathscr{B}_1(Q_0)$ of $Q_0$, and the secondary stopping cubes $\mathscr{C}_1=\mathscr{C}_1(Q_0)$ of $Q_0$ as explained above.
  \item Assuming that the collections $\mathscr{B}_k$ and $\mathscr{C}_k$ are already constructed, for every  $Q\in\mathscr{C}_k$, we choose (using $b^1_Q$) the primary stopping cubes $\mathscr{B}_1(Q)$ of $Q$, which determine the functions $\tilde{b}^1_Q$ and $e^1_Q$, and then  (using $\tilde{b}^1_Q$ and $e^1_Q$), the  secondary stopping cubes $\mathscr{C}_1(Q)$ of $Q$. We let
\begin{equation*}
  \mathscr{B}_{k+1}:=\bigcup_{Q\in\mathscr{C}_k}\mathscr{B}_1(Q),\qquad  \mathscr{C}_{k+1}:=\bigcup_{Q\in\mathscr{C}_k}\mathscr{C}_1(Q).
\end{equation*}
\end{itemize}
By iterating the bounds $\sum_{Q'\in\mathscr{B}_1(Q)}\abs{Q'}\leq\delta\abs{Q}$ and $\sum_{Q'\in\mathscr{C}_1(Q)}\abs{Q'}\leq(1-\tau)\abs{Q}$, we arrive at
\begin{equation*}
  \sum_{Q\in\mathscr{C}_k(Q_0)}\abs{Q}\leq(1-\tau)^k\abs{Q_0},\qquad
  \sum_{Q\in\mathscr{B}_k(Q_0)}\abs{Q}\leq\delta\sum_{Q\in\mathscr{C}_{k-1}(Q_0)}\abs{Q}\leq\delta(1-\tau)^{k-1}\abs{Q_0},
\end{equation*}
where we interpret $\mathscr{C}_0(Q_0):=\{Q_0\}$.
Hence the measure of all primary stopping cubes satisfies
\begin{equation*}
   \sum_{k=1}^\infty\sum_{Q\in\mathscr{B}_k(Q_0)}\abs{Q}
   \leq\sum_{k=1}^\infty\delta(1-\tau)^{k-1}\abs{Q_0}
   =\frac{\delta}{\tau}\abs{Q_0}.
\end{equation*}
The parameter $\delta$ can be made small independently of $\tau$, and hence we can make the fraction $\delta/\tau$ as small as we like.

Then we can define
\begin{equation*}
  \Phi(x):=\sup\Big\{\dist(x,(3Q)^c):Q\in\bigcup_{k=1}^\infty\mathscr{B}_k\Big\}.
\end{equation*}
It follows that
\begin{equation}\label{eq:PhiPosSmall}
  \abs{\{\Phi>0\}}=\Babs{\bigcup_{k=1}^\infty\bigcup_{Q\in\mathscr{B}_k}3Q}
  \leq 3^d\frac{\delta}{\tau}\abs{Q_0}=:\varrho\abs{Q_0},
\end{equation}
where the fraction $\varrho$ can be made arbitrarily small.

\subsection{Construction summary; completion of the proof of Proposition~\ref{prop:testTPhi}}
Given a cube $Q_0$, we find the stopping cubes $\bigcup_{k=1}^\infty\mathscr{B}_k(Q_0)$ and $\bigcup_{k=1}^\infty\mathscr{C}_k(Q_0)$. Let us also call $Q_0$ itself a stopping cube, and denote $\mathscr{C}_0(Q_0):=\{Q_0\}$. The stopping cubes determine the function $\Phi$. For each $Q\in\mathscr{C}(Q_0):=\bigcup_{k=0}^\infty\mathscr{C}_k(Q_0)$, there is a function $\tilde{b}^1_Q$, the good part of the Calder\'on--Zygmund decomposition of $b^1_Q$, thus
\begin{equation*}
  \Norm{\tilde{b}^1_Q}{\infty}\lesssim 1.
\end{equation*}

For every $Q\in\mathscr{C}(Q_0)$, we can apply the estimate \eqref{eq:TPhiTildeB} for $Q$ in place of $Q_0$. Indeed, it suffices to check that the chosen $\Phi$ satisfies the analogue of \eqref{eq:PhiLower} with $Q\in\mathscr{C}(Q_0)$ in place of $Q_0$, namely, that
\begin{equation*}
  \Phi(x)\geq\sup_{Q'\in\mathscr{B}_1(Q)}\dist(x,(3Q')^c).
\end{equation*}
But this is clear from the definition of $\Phi$, since $\mathscr{B}_1(Q)\subseteq\mathscr{B}(Q_0)$ for all $Q\in\mathscr{C}(Q_0)$, and $\Phi$ is the supremum over this larger set.
Thus, by \eqref{eq:TPhiTildeB} applied to $Q$ in place of $Q_0$, we have
\begin{equation*}
  \abs{T_\Phi\tilde{b}^1_Q}\lesssim 1+e^1_Q.
\end{equation*}
If $Q'\subseteq Q\in\mathscr{C}(Q_0)$ is not contained in any smaller $Q''\in\mathscr{C}(Q_0)$, equivalently, not in any $Q''\in\mathscr{C}_1(Q)$,
then, by the construction of the secondary stopping cubes $\mathscr{C}_1(Q)$, this means that
\begin{equation}\label{eq:TPhiBu}
  \fint_{Q'}\abs{T_\Phi\tilde{b}^1_Q}^u \lesssim \fint_{Q'}(1+e^1_Q)^u\lesssim 1,
\end{equation}
and
\begin{equation*}
  \Babs{\fint_{Q'}\tilde{b}^1_Q} \gtrsim 1,\qquad   
   Q'\subseteq Q\in\mathscr{C}(Q_0),\quad Q'\not\subseteq Q''\in\mathscr{C}_1(Q).
\end{equation*}
(We are suppressing the dependence on the parameters $\varepsilon,\eta$, since they are now fixed and no longer relevant to us.) Recall also that
\begin{equation*}
  \sum_{Q'\in\mathscr{C}_1(Q)}\abs{Q'}\leq (1-\tau)\abs{Q}.
\end{equation*}

Summa summarum, associated to every secondary stopping cube $Q\in\mathscr{C}(Q_0)$, there is a non-degenerate test function $\tilde{b}^1_Q\in L^\infty(Q)$ such that $1_Q T_\Phi\tilde{b}^1_Q\in L^u(Q)$, with the correct normalization.
Moreover,
\begin{equation*}
  \abs{\{\Phi>0\}}\leq\varrho\abs{Q_0}.
\end{equation*}

Of course, starting from the original test functions $b^2_Q$ and $T^*$ in place of $T$, we can similarly produce new test functions $\tilde{b}^2_Q\in L^\infty(Q)$ with $1_Q T^*_\Phi\tilde{b}^2_Q\in L^u(Q)$ for another set of secondary stopping cube $Q\in\mathscr{C}^2(Q_0)$. To have the same $\Phi$ both for $T$ and $T^*$, we should define
\begin{equation*}
 \Phi(x):=\sup\{\dist(x,(3Q)^c):Q\in\bigcup_{k=1}^\infty(\mathscr{B}_k^1\cup\mathscr{B}^2_k)\},
\end{equation*}
where $\mathscr{B}^1_k$ and $\mathscr{B}^2_k$ are the primary stopping cubes related to $b^1_Q$ and $b^2_Q$, respectively.
Clearly, this still satisfies the bound \eqref{eq:PhiPosSmall} , with at most twice the original constant $\varrho$, which we can make arbitarily small in any case.

To check that the $(\infty,u)$ accretive system $\tilde{b}^1_Q$ for $T_\Phi$, on sparse subcubes of $Q_0$, is strong, by Lemma~\ref{lem:strongAccr}, we only need to check that there is a $(u',1)$ accretive system for $T^*_\Phi$ on all dyadic cubes. And we claim that such a system is provided by the original $(p,p)$ accretive system $b^2_Q$ for~$(T^*)_{\#}$, {whose existence is contained in the very assumptions of Proposition \ref{prop:testTPhi} under which we are working in this Section \ref{sec:modTestF}.} By Lemma~\ref{lem:Tsharp2TPhi}, $b^2_Q$ is also a $(p,p)$ accretive system for $T^*_\Phi$, and thus a $(u',1)$ accretive system provided that $u'\leq p$. But we can choose $u$ as large as we like, hence $u'$ as close to $1$ as we like, and then this condition is clearly fulfilled.

This completes the proof of Proposition~\ref{prop:testTPhi}.

\section{The baby $Tb$ theorem; proof of Proposition~\ref{prop:babyTb}}

Let us denote the reference cube in which we operate by $Q^0$, as we will need the notation $Q_0$ for another purpose below.
Let us first deal with just one accretive system $b_Q$ defined on a sparse family $\mathscr{Q}$; later on, these results will be applied to both $b^1_Q$ on $\mathscr{Q}_1$ and $b^2_Q$ on $\mathscr{Q}_2$. We also refer to the members of $\mathscr{Q}$ as \emph{stopping cubes}.
For every $Q\subseteq Q^0$, let $Q^a$ be the minimal stopping cube that contains $Q$. Then
\begin{equation*}
  \Babs{\fint_Q b_{Q_a}\ud x}\gtrsim 1,\quad
  \fint_Q\abs{b_{Q_a}}^{{t}}\ud x\lesssim 1,\quad
  \fint_Q\abs{Tb_{Q_a}}^{{t}}\ud x\lesssim 1,
\end{equation*}
{for some exponent $t\in(1,\infty)$ as in the assumptions of Proposition \ref{prop:babyTb}.}

We start by recalling the adapted martingale difference framework for a local $Tb$ theorem from \cite{NTV:Duke} and \cite{HyMa:localTb,HyVa}. (Also the subsequent analysis borrows from these papers, even when this is not always stated. On the other hand, we take the opportunity to simplify several details, as we are in the simpler case of the Lebesgue measure, rather than a non-doubling one; this allows us to work with the fixed system of standard dyadic cubes, instead of their random translation.)

We have the expectation (averaging) operators
\begin{equation*}
  \mathbb{E}_Q^b f:=1_Q b_{Q^a}\frac{\ave{f}_Q}{\ave{b_{Q^a}}_Q},\qquad
  \mathbb{E}_Q:=1_Q\ave{f}_Q,
\end{equation*}
and the difference operators
\begin{equation*}
\begin{split}
  \mathbb{D}_Q^b f
  &:=\sum_{i=1}^{2^d} \mathbb{E}_{Q_i}^b f-\mathbb{E}_Q^b f \\
  &=\sum_{i=1}^{2^d} \Big(1_{Q_i} b_{Q_i^a}\frac{1}{\ave{b_{Q_i^a}}_{Q_i}}-1_Q b_{Q^a}\frac{\abs{Q_i}}{\abs{Q}\ave{b_{Q^a}}_Q}\Big)\ave{f}_{Q_i}
  =:\sum_{i=1}^{2^d}\phi_{Q,i}\ave{f}_{Q_i},
\end{split}
\end{equation*}
where the $i$-sum goes through the dyadic children $Q_i$ of $Q$. A direct computation shows that
\begin{equation*}
  (\mathbb{D}_Q^b)^2 f=\mathbb{D}_Q^b f-\sum_{i=1}^{2^d} 1_{Q_i}\Big(\frac{\ave{b_{Q^a}}_{Q_i}}{\ave{b_{Q_i^a}}_{Q_i}}b_{{Q_i^a}}-b_{Q^a}\Big)
     \frac{\ave{f}_Q}{\ave{b_{Q^a}}_Q}=:\mathbb{D}_Q^b f-\omega_Q\ave{f}_Q.
\end{equation*}
Note that $1_{Q_i}\omega_Q$ is nonzero only if $Q_i^a\neq Q^a$, i.e., only if $Q_i$ is a stopping cube. Thus $\omega_Q$ is nonzero only if $Q$ has at least one stopping child. Combining the above formulae, we get
\begin{equation}\label{eq:DQformula}
  \mathbb{D}_Q^b f
  =(\mathbb{D}_Q^b)^2f+\omega_Q\ave{f}_Q
  =\sum_{i=1}^{2^d}\phi_{Q,i}\ave{\mathbb{D}_Q^b f}_{Q_i}+\omega_Q \ave{f}_Q
  =:\sum_{i=0}^{2^d}\phi_{Q,i}\ave{\mathbb{D}_{Q,i}^b f}_{Q_i},
\end{equation}
where $Q_0:=Q$, $\phi_{Q,0}:=\omega_Q$ and
\begin{equation*}
  \mathbb{D}_{Q,i}^b f:=\begin{cases} \mathbb{D}_Q^b f, & \text{if }i = 1,\ldots, 2^d,\\ \mathbb{E}_Q f, & \text{if $i=0$ and $Q$ has a stopping child}, \\
     0, & \text{if $i=0$ and $Q$ does not have any stopping children}. \end{cases}
\end{equation*}
We have the following important estimates. The $L^2$ case is from \cite{NTV:Duke}, and its generalization to $L^r$ from \cite{HyVa}. (Both these papers deal with more general non-doubling situations, the latter even vector-valued---a generality that we do not consider here.)

\begin{proposition}\label{prop:LP}
For all $r\in(1,\infty)$, $i=0,1,\ldots,2^d$, and $f\in L^r$, we have
\begin{equation*}
    \BNorm{\Big(\sum_Q\abs{\mathbb{D}_{Q,i}^b f}^2\Big)^{1/2}}{r}
    +\BNorm{\Big(\sum_Q\abs{(\mathbb{D}_{Q,i}^b)^* f}^2\Big)^{1/2}}{r}\lesssim\Norm{f}{r}.
\end{equation*}
In particular,
\begin{equation*}
    \Big(\sum_Q\Norm{\mathbb{D}_{Q,i}^b f}{2}^2\Big)^{1/2}
    +\Big(\sum_Q\Norm{(\mathbb{D}_{Q,i}^b)^* f}{2}^2\Big)^{1/2}\lesssim\Norm{f}{2}.
\end{equation*}
\end{proposition}

For every $f\in L^r(Q^0)$, we have the decomposition
\begin{equation*}
  f=\mathbb{E}_{Q^0}^b f+\sum_Q\mathbb{D}_Q^b f.
\end{equation*}
{Here and below, we apply the implicit convention that all summations over $Q$ (or other variables designating cubes) always carry the restriction $Q\subseteq Q^0$, in addition to possible explicit summation conditions.}
To simplify writing, we redefine $\mathbb{D}_{Q^0}^b f$ as $\mathbb{E}_{Q^0}^b f+\mathbb{D}_{Q^0}^b f$, so that we can drop the first term from the sum above. Thus
\begin{equation*}
  \pair{Tf}{g}=\sum_{Q,R}\pair{T\mathbb{D}_Q^{b_1}f}{\mathbb{D}_R^{b_2}g}
  =\sum_{\substack{Q,R\\ \ell(Q)\leq\ell(R)}}+\sum_{\substack{Q,R\\ \ell(Q)>\ell(R)}}.
\end{equation*}
By symmetry, it suffices to estimate the first half. This we reorganize as follows:
\begin{equation*}
\begin{split}
   \sum_{\substack{Q,R\\ \ell(Q)\leq\ell(R)}}\pair{T\mathbb{D}_Q^{b_1}f}{\mathbb{D}_R^{b_2}g}
   &=\sum_{k=0}^\infty\sum_{m\in\Z^d}\sum_R\sum_{\substack{Q\subseteq R+\ell(R)m\\ \ell(Q)=2^{-k}\ell(R)}}\pair{T\mathbb{D}_Q^{b_1}f}{\mathbb{D}_R^{b_2}g} \\
   &=\sum_{k=0}^\infty\sum_{m\in\Z^d}\sum_R\pair{T\mathbb{D}_{R+\ell(R)m}^{b_1,k}f}{\mathbb{D}_R^{b_2}g},
\end{split}
\end{equation*}
where $(S:=R+\ell(R)m)$
\begin{equation*}
     \mathbb{D}_{S}^{b_1,k}f :=\sum_{\substack{Q\subseteq S\\ \ell(Q)=2^{-k}\ell(S)}}\mathbb{D}_{Q}^{b_1}f.
\end{equation*}
Below, we will also use the notation $\mathbb{D}_{S,i}^{b_1,k}$ (with additional subscript $i\in\{0,1,\ldots,2^d\}$) similarly defined with $\mathbb{D}_{Q,i}^{b_1}$ on the right as well.

{We now turn to the estimation of the above series in various parts. For some of them, we will achieve domination by the usual $L^2$ bound
\begin{equation*}
  \Norm{f}{2}\Norm{g}{2}\leq\abs{Q^0}^{1/2-1/s'}\Norm{f}{s'}\abs{Q^0}^{1/2-1/s'}\Norm{g}{s'}
  =\abs{Q^0}^{1-2/s'}\Norm{f}{s'}\Norm{g}{s'},
\end{equation*}
but the weaker right-hand side, as in the statement of Proposition \ref{prop:babyTb}, will appear in some cases.}

\subsection{Disjoint cubes}
We further analyse the part of the $m$ sum with $m\neq 0$, thus $Q\subseteq R+m\ell(R)$ is disjoint from $R$.
By \eqref{eq:DQformula}, we can expand
\begin{equation*}
\begin{split}
  \pair{T\mathbb{D}_Q^{b_1}f}{\mathbb{D}_R^{b_2}g}
  &=\sum_{i,j}\ave{\mathbb{D}_{Q,i}^{b_1}f}_{Q_i}\pair{T\phi_{Q,i}^1}{\phi_{R,j}^2}\ave{\mathbb{D}_{R,j}^{{b_2}}g}_{R_j},\qquad\text{hence, summing in }Q, \\
   \pair{T\mathbb{D}_S^{b_1,k}f}{\mathbb{D}_R^{b_2}g} 
   &=\sum_{i,j}\iint_{R\times S}\mathbb{D}_{S,i}^{b_1,k}f(y)K_{R,S}^{i,j;k}(x,y)\mathbb{D}_{R,j}^{b_2}g(x)\ud x\ud y,\qquad\text{where}\\
    K_{R,S}^{i,j;k}(x,y)
   &:=\sum_{\substack{Q\subseteq S\\ \ell(Q)=2^{-k}\ell(S)}}
    \frac{1_{Q_i}(y)}{\abs{Q_i}}\pair{T\phi_{Q,i}^1}{\phi_{R,j}^2}\frac{1_{R_j}(x)}{\abs{R_j}}.
\end{split}
\end{equation*}

{
\begin{lemma}\label{lem:disjointCoef}
For $m\neq 0$, we have
\begin{equation*}
  \Norm{K^{i,j;k}_{R,R+\ell(R)m}}{L^2(\R^d\times\R^d)}\lesssim
    2^{-k\alpha}\abs{m}^{-d-\alpha},
\end{equation*}
when the H\"older exponent in the Calder\'on--Zygmund condition \eqref{eq:CZK} is $\alpha\in(0,\frac12)$.
\end{lemma}

Since we can always decrease the H\"older exponent of the Calder\'on--Zygmund kernel, we will henceforth assume that $\alpha<\tfrac12$.}

\begin{proof}
Note that $Q\neq Q^0$ in this sum. Indeed, $\ell(Q)\leq\ell(R)\leq\ell(Q^0)$ so the only way that we could have $Q=Q^0$ is $\ell(Q)=\ell(R)=\ell(Q^0)$, and then (since $Q^0$  is the only cube of sidelength $\ell(Q^0)$), $Q=R=Q^0$. But then $m=0$, a contradiction. Thus $\mathbb{D}_Q^{b_1}$ is always given by the original definition, i.e., without the addition of $\mathbb{E}_{Q^0}^{b_1}$, and then all the $\phi_{Q,i}^1$ have mean zero. Hence we can write
\begin{equation*}
\begin{split}
  \pair{T\phi_{Q,i}^1}{\phi_{R,j}^2}
  &=\iint_{R\times Q}K(x,y)\phi_{Q,i}^1(y)\phi_{R,j}^2(x)\ud x\ud y \\
  &=\iint_{R\cap 3Q \times Q}K(x,y)\phi_{Q,i}^1(y)\phi_{R,j}^2(x)\ud x\ud y \\
  &\qquad +\iint_{R\setminus(3Q)\times Q}[K(x,y)-K(x,c_Q)]\phi_{Q,i}^1(y)\phi_{R,j}^2(x)\ud x\ud y,
\end{split}
\end{equation*}
and then estimate
\begin{equation*}
\begin{split}
  \abs{\pair{T\phi_{Q,i}^1}{\phi_{R,j}^2}}
  &\lesssim\iint_{R\cap 3Q \times Q}\frac{\ud x\ud y}{\abs{x-y}^d}+\iint_{R\setminus(3Q)\times Q}
     \frac{\ell(Q)^\alpha}{\abs{x-c_Q}^{d+\alpha}}\ud x\ud y.
\end{split}
\end{equation*}

If $\abs{m}_\infty>1$, the first term vanishes, and estimating the second term we get
\begin{equation*}
  \abs{\pair{T\phi_{Q,i}^1}{\phi_{R,j}^2}}
  \lesssim\frac{\ell(Q)^\alpha}{\dist(Q,R)^{d+\alpha}}\abs{Q}\abs{R}
  \lesssim 2^{-k\alpha}\abs{m}^{-d-\alpha}\abs{Q},
\end{equation*}
and then $\Norm{K^{i,j;k}_{R,R+\ell(R)m}}{\infty}\lesssim 2^{-k\alpha}\abs{m}^{-d-\alpha}/\abs{R}$.

If $\abs{m}_\infty=1$, then the first term is nonzero (and then bounded by ${C}\abs{Q}$) only if $\dist(Q,R)=0$, while the second term is estimated by
\begin{equation*}
 \frac{\ell(Q)^\alpha \abs{Q}}{\max\{\ell(Q),\dist(Q,R)\}^\alpha}. 
\end{equation*}
So altogether we have
\begin{equation*}
  \abs{\pair{T\phi_{Q,i}^1}{\phi_{R,j}^2}}\lesssim \Big(1+\frac{\dist(Q,R)}{\ell(Q)}\Big)^{-\alpha}\abs{Q},
\end{equation*}
and then
\begin{equation*}
  \abs{K^{i,j;k}_{R,R+\ell(R)m}(x,y)}\lesssim \frac{1_R(x)}{\abs{R}} \sum_{\substack{Q\subseteq R+\ell(R)m\\ \ell(Q)=2^{-k}\ell(R)}}
      \Big(1+\frac{\dist(Q,R)}{\ell(Q)}\Big)^{-\alpha}1_Q(y).
\end{equation*}
The number of the cubes $Q$ with $\dist(Q,R)=n\ell(Q)$, $(n=0,1,\ldots,2^k-1)$, is $\mathcal{O}(2^{k(d-1)})$, while each of them has measure $\abs{Q}=2^{-kd}\abs{R}$. Hence
\begin{equation*}
  \Norm{K^{i,j;k}_{R,R+\ell(R)m}}{L^2(\R^d\times\R^d)}^2
  \lesssim\frac{1}{\abs{R}}\sum_{n=0}^{2^k-1} 2^{k(d-1)}\cdot (1+n)^{-2\alpha}\cdot 2^{-kd}\abs{R}
  {\eqsim 2^{-2k\alpha}.}\qedhere
\end{equation*}
\end{proof}

{(We note that it is the last step of the previous computation that would lead to a slightly different conclusion for $\alpha\geq\frac12$.)}

Lemma~\ref{lem:disjointCoef} is enough to estimate the part of the series with $m\neq 0$; indeed
\begin{equation*}
\begin{split}
  \sum_{m\neq 0} &\sum_{k=0}^\infty \sum_R\abs{\pair{T\mathbb{D}_{R+\ell(R)m}^{b_1,k}f}{\mathbb{D}_R^{b_2}g}} \\
  &\leq\sum_{m\neq 0}\sum_{k=0}^\infty\sum_R\sum_{i,j}\Norm{K_{R,{R+\ell(R)m}}^{i,j;k}}{2}\Norm{\mathbb{D}_{R+\ell(R)m}^{b_1,k}f}{2}\Norm{\mathbb{D}_R^{b_2}g}{2} \\
  &\lesssim\sum_{m\neq 0}\abs{m}^{-d-\alpha}\sum_{k=0}^\infty 2^{-\alpha k}
  \sum_R\Norm{\mathbb{D}_{R+\ell(R)m}^{b_1,k}f}{2}\Norm{\mathbb{D}_R^{b_2}g}{2},
\end{split}
\end{equation*}
where
\begin{equation*}
\begin{split}
  \sum_R &\Norm{\mathbb{D}_{R+\ell(R)m}^{b_1,k}f}{2}\Norm{\mathbb{D}_R^{b_2}g}{2} \\
  &\leq \Big(\sum_R\Norm{\mathbb{D}_{R+\ell(R)m}^{b_1,k}f}{2}^2\Big)^{1/2}\Big(\sum_R\Norm{\mathbb{D}_R^{b_2}g}{2}^2\Big)^{1/2} \\
  &=\Big(\sum_Q\Norm{\mathbb{D}_{Q}^{b_1}f}{2}^2\Big)^{1/2}\Big(\sum_R\Norm{\mathbb{D}_R^{b_2}g}{2}^2\Big)^{1/2}
  \lesssim\Norm{f}{2}\Norm{g}{2},
\end{split}
\end{equation*}
and
\begin{equation*}
  \sum_{m\neq 0}\abs{m}^{-d-\alpha}\sum_{k=0}^\infty 2^{-\alpha k}\lesssim 1.
\end{equation*}

\subsection{Nested cubes}
We are left with the part with $m=0$, that is,
\begin{equation*}
  \sum_{k=0}^\infty\sum_R\pair{T\mathbb{D}_R^{b_1,k}f}{\mathbb{D}_R^{b_2}g}
  =\sum_R\pair{T\mathbb{D}_R^{b_1}f}{\mathbb{D}_R^{b_2}g}+\sum_R\sum_{Q\subsetneq R}\pair{T\mathbb{D}_Q^{b_1}f}{\mathbb{D}_R^{b_2}g}.
\end{equation*}

\begin{lemma}
We have
\begin{equation}\label{eq:nestedPart}
\begin{split}
  \sum_R\sum_{Q\subsetneq R}\pair{T\mathbb{D}_Q^{b_1}f}{\mathbb{D}_R^{b_2}g}
  &=\sum_Q\pair{T\mathbb{D}_Q^{b_1}f}{b^2_{Q^{a,2}}}\frac{\ave{g}_Q}{\ave{b_{Q^{a,2}}^2}_Q}
   { -\pair{Tf}{b^2_{Q^0}}
    \frac{\ave{g}_{Q^0}}{\ave{b^2_{Q^0}}_{Q^0}}} \\
  &\qquad+\sum_{k=0}^\infty\sum_R\sum_{\substack{S\subseteq R \\ \ell(S)=\ell(R)/2}}
     \sum_{j=0}^{2^d}\pair{T\mathbb{D}_S^{b_1,k}f}{1_{S^c}\psi_{R,j;S}^{2}}\ave{\mathbb{D}_{R,j}^{b_2}g}_{R_j},
\end{split}
\end{equation}
for some bounded functions $\psi_{R,j;S}^{2}$.
\end{lemma}

\begin{proof}
For $Q\subsetneq R$, let $R_Q$ be the unique subcube of $R$ {that} contains $Q$. Then
\begin{equation*}
  \mathbb{D}_R^{b_2}g
  =1_{R_Q}\mathbb{D}_R^{b_2}g+ 1_{R_Q^c}\mathbb{D}_R^{b_2}g,
\end{equation*}
where further (we temporarily drop the superscript $2$)
\begin{equation*}
  1_{R_Q}\mathbb{D}_R^{b}g
  =(1-1_{R_Q^c})\Big(b_{R_Q^{a}}\frac{\ave{{g}}_{R_Q}}{\ave{b_{R_Q^{a}}}_{R_Q}}
    -b_{R^{a}}\frac{\ave{{g}}_{R}}{\ave{b_{R^{a}}}_{R}}\Big)
\end{equation*}
and
\begin{equation}\label{eq:nestRewrite}
\begin{split}
  b_{R_Q^{a}} &\frac{\ave{{g}}_{R_Q}}{\ave{b_{R_Q^{a}}}_{R_Q}}
    -b_{R^{a}}\frac{\ave{{g}}_{R}}{\ave{b_{R^{a}}}_{R}} \\
  &=\frac{b_{R_Q^{a}}}{\ave{b_{R_Q^a}}_{R_Q}}
  \Big(\ave{{g}}_{R_Q}-\ave{b_{R^a}}_{R_Q}\frac{\ave{{g}}_{R}}{\ave{b_{R^{a}}}_{R}}\Big)
  +\Big(\frac{\ave{b_{R^a}}_{R_Q}}{\ave{b_{R_Q^a}}_{R_Q}}b_{R_Q^a}-b_{R^{a}}\Big)\frac{\ave{{g}}_{R}}{\ave{b_{R^{a}}}_{R}} \\
  &=\frac{b_{R_Q^{a}}}{\ave{b_{R_Q^a}}_{R_Q}}\ave{\mathbb{D}_R^b {g}}_{R_Q}
  +\Big(\frac{\ave{b_{R^a}}_{R_Q}}{\ave{b_{R_Q^a}}_{R_Q}}b_{R_Q^a}-b_{R^{a}}\Big)
  \frac{\ave{\mathbb{D}_{R,0}^b {g}}_{R}}{\ave{b_{R^{a}}}_{R}}.
\end{split}
\end{equation}
On the last line we observed that the function in the parentheses is zero unless $R_Q$ is a stopping cube (i.e., $R_Q^a=R_Q\neq R^a$), and thus we may replace ${g}$ by $\mathbb{D}_{R,0}^b {g}$ inside the average on $R$.

Substituting back, it follows that
\begin{equation}\label{eq:babyTbAStep}
\begin{split}
   \mathbb{D}_R^{b}g
   &=   b_{R_Q^{a}}\frac{\ave{{g}}_{R_Q}}{\ave{b_{R_Q^{a}}}_{R_Q}}
    -b_{R^{a}}\frac{\ave{{g}}_{R}}{\ave{b_{R^{a}}}_{R}} \\
   &\quad +1_{R_Q^c}\Big\{\mathbb{D}_R^{b}g
       -\frac{b_{R_Q^{a}}}{\ave{b_{R_Q^{a}}}_{R_Q}}\ave{\mathbb{D}_R^{b}g}_{R_Q}
        -\Big(\frac{\ave{b_{R^a}}_{R_Q}}{\ave{b_{R_Q^a}}_{R_Q}}b_{R_Q^a}-b_{R^{a}}\Big)
        \frac{\ave{\mathbb{D}_{R,0}^b {g}}_{R}}{\ave{b_{R^{a}}}_{R}}
        \Big\}  \\
   &=\Big(b_{R_Q^{a}}\frac{\ave{g}_{R_Q}}{\ave{b_{R_Q^{a}}}_{R_Q}}
    -b_{R^{a}}\frac{\ave{g}_{R}}{\ave{b_{R^{a}}}_{R}}\Big)
    +1_{R_Q^c} {\sum_{j=0}^{2^d}}\psi_{R,j;R_Q}^{{2}}\ave{\mathbb{D}_{R,j}^{b}g}_{R_j},
\end{split}
\end{equation}
where (recalling the formula $\mathbb{D}_R^b g={\sum_{j=0}^{2^d}}\phi_{R,j}\ave{\mathbb{D}_{R,j}^b g}_{R_j}$)
\begin{equation*}
  \psi_{R,j;S}^{{2}}:=\phi_{R,j}-
    \begin{cases} \Big(\ave{b_{R^a}}_{S}/\ave{b_{S^a}}_{S}\cdot b_{S^a}
         -b_{R^{a}}\Big)\Big/\ave{b_{R^{a}}}_{R}, & \text{if }j=0,\\
     b_{S^{a}}/\ave{b_{S^{a}}}_{S}, & \text{if }j\geq 1\text{ and }R_j=S, \\
     0. & \text{else}
     \end{cases}
\end{equation*}
are bounded functions.

Pairing with $T\mathbb{D}_Q^{b_1}f$, we obtain (changing the summation order, and observing that $Q$ is the smallest $R_Q$, as well as the telescoping cancellation)
\begin{equation*}
\begin{split}
  \sum_Q\sum_{R\supsetneq Q} &\Bpair{T\mathbb{D}_Q^{b_1}f}{b^2_{R_Q^{a,2}}\frac{\ave{g}_{R_Q}}{\ave{b^2_{R_Q^{a,2}}}_{R_Q}}
    -b^2_{R^{a,2}}\frac{\ave{g}_{R}}{\ave{b^2_{R^{a,2}}}_{R}}} \\
    &=\sum_Q \pair{T\mathbb{D}_Q^{b_1}f}{b^2_{Q^{a,2}}}
    \frac{\ave{g}_{Q}}{\ave{b^2_{{Q}^{a,2}}}_{Q}}
   { -\sum_Q \pair{T\mathbb{D}_Q^{b_1}f}{b^2_{Q^0}}
    \frac{\ave{g}_{Q^0}}{\ave{b^2_{Q^0}}_{Q^0}},}
\end{split}
\end{equation*}
{where the first term on the right arises from $R_Q=Q$ when $R\supsetneq Q$ takes its minimal value, the dyadic parent of $Q$, while the second term arises from the maximal value $R=Q^0$ coming from the implicit summation condition $R\subseteq Q^0$. Since the initial cube $Q^0$ is defined to be a stopping cube, we have $(Q^0)^{a,2}=Q^0$. Recalling that $f=\sum_Q\mathbb{D}_Q^{b_1}f$, the second term on the right simplifies to
\begin{equation*}
  \sum_Q \pair{T\mathbb{D}_Q^{b_1}f}{b^2_{Q^0}}
    \frac{\ave{g}_{Q^0}}{\ave{b^2_{Q^0}}_{Q^0}}
    =\pair{Tf}{b^2_{Q^0}}
    \frac{\ave{g}_{Q^0}}{\ave{b^2_{Q^0}}_{Q^0}}.
\end{equation*}}

For the other term on the right of \eqref{eq:babyTbAStep}, we introduce the auxiliary summation variable $S:=R_Q$, regroup the summation according to the relative size of $Q$ and $S$, and recall the notation
\begin{equation*}
  \mathbb{D}_S^{b_1,k}f:=\sum_{\substack{Q\subseteq S\\ \ell(Q)=2^{-k}\ell(S)}}\mathbb{D}_Q^{b_1}f
\end{equation*}
to arrive at the asserted formula.
\end{proof}

{
The middle term on the right of \eqref{eq:nestedPart} is estimated by the assumption that $b_Q^2$ is a (strong) $(\infty,t)$ accretive system for $T^*$ on sparse subcubes of $Q^0$, namely
\begin{equation*}
  \abs{\pair{Tf}{b_{Q^0}^2}}
  =\abs{\pair{f}{T^*b_{Q^0}^2}}
  =\abs{Q^0}\Babs{\fint_{Q^0} f\cdot T^*b_{Q^0}^2\ud x}
  \leq\abs{Q^0}\Big(\fint_{Q^0}\abs{f}^{t'}\ud x\Big)^{1/t'}\Big(\fint_{Q^0}\abs{T^*(b_{Q^0}^2)}^{t}\ud x\Big)^{1/t},
\end{equation*}
and hence (recalling that $s'>t'>1$)
\begin{equation*}
\begin{split}
  \Babs{ \pair{Tf}{b_{Q^0}^2} \frac{\ave{g}_{Q^0}}{\ave{b^2_{Q^0}}_{Q^0}} }
  &\lesssim\abs{Q^0}\Big(\fint_{Q^0}\abs{f}^{t'}\ud x\Big)^{1/t'}\fint_{Q^0}\abs{g}\ud x \\
  &\leq\abs{Q^0}\Big(\fint_{Q^0}\abs{f}^{s'}\ud x\Big)^{1/s'}\Big(\fint_{Q^0}\abs{g}^{s'}\ud x\Big)^{1/s'}
  =\abs{Q^0}^{1-2/s'}\Norm{f}{s'}\Norm{g}{s'},
\end{split}
\end{equation*}
which agrees with the right hand side of the bound in Proposition \ref{prop:babyTb}.}

The last summand in \eqref{eq:nestedPart} can be written as
\begin{equation*}
\begin{split}
  \pair{T\mathbb{D}_S^{b_1,k}f}{1_{S^c}\psi_{R,j;S}^{{2}}}\ave{\mathbb{D}_{R,j}^{{b_2}}g}_{R_j}
  &=\sum_{\substack{Q\subseteq S;\,i=0,\ldots,2^d\\ \ell(Q)=2^{-k}\ell(S)}}
     \ave{\mathbb{D}_{Q,i}^{b_1}f}_{Q_i}
     \pair{T\phi_{Q,i}^1}{1_{S^c}\psi_{R,j;S}^{{2}}}\ave{\mathbb{D}_{R,j}^{b_2}g}_{R_j} \\
  &={\sum_{i=0}^{2^d}}
  \iint \mathbb{D}_{S,i}^{b_1,k}f(y)\tilde{K}^{i,j;k}_{R,S}(x,y)\mathbb{D}_{R,j}^{b_2}g(x)\ud x\ud y,\qquad\text{where}\\
  \tilde{K}^{i,j;k}_{R,S}(x,y) &:=\sum_{{\substack{Q\subseteq S\\ \ell(Q)=2^{-k}\ell(S)}}}\frac{1_{Q_i}(y)}{\abs{Q_i}}
     \pair{T\phi_{Q,i}^1}{1_{S^c}\psi_{R,j;S}^{2}}\frac{1_{R_j}(x)}{\abs{R_j}}.
\end{split}
\end{equation*}
Just as in Lemma~\ref{lem:disjointCoef} (case $\abs{m}_\infty=1$) we check that:

{
\begin{lemma}\label{lem:tildeK}
\begin{equation*}
  \Norm{\tilde{K}^{i,j;k}_{R,S}}{L^2(\R^{2d})}\lesssim 2^{-k\alpha},\qquad\alpha<\tfrac12.
\end{equation*}
\end{lemma}

\begin{proof}
We first estimate the coefficients $\pair{T\phi_{Q,i}^1}{1_{S^c}\psi_{R,j;S}^{2}}$. Using the mean zero property of $\phi_{Q,i}^1$,
\begin{equation*}
\begin{split}
  \pair{T\phi_{Q,i}^1}{1_{S^c}\psi_{R,j;S}^{2}}
  &=\iint_{S^c\times Q}K(x,y)\phi_{Q,i}^1(y)\psi_{R,j;S}^{2}(x)\ud x\ud y \\
  &=\iint_{S^c\cap 3Q \times Q}K(x,y)\phi_{Q,i}^1(y)\psi_{R,j;S}^{2}(x)\ud x\ud y \\
  &\qquad +\iint_{S^c\setminus(3Q)\times Q}[K(x,y)-K(x,c_Q)]\phi_{Q,i}^1(y)\psi_{R,j;S}^{2}(x)\ud x\ud y.
\end{split}
\end{equation*}
By the boundedness of both $\phi_{Q,i}^1$ and $\psi_{R,j;S}^{2}$, and the fact that $Q\subseteq S\subseteq R$ so that $S^c\subseteq Q^c$, we have
\begin{equation*}
\begin{split}
 \abs{\pair{T\phi_{Q,i}^1}{1_{S^c}\psi_{R,j;S}^{2}}}
  &\lesssim\iint_{(3Q\setminus Q) \times Q}\frac{\ud x\ud y}{\abs{x-y}^d}+\iint_{S^c\setminus(3Q)\times Q}
     \frac{\ell(Q)^\alpha}{\abs{x-c_Q}^{d+\alpha}}\ud x\ud y.
\end{split}
\end{equation*}
The first term is bounded by $C\abs{Q}$ and the second is estimated as
\begin{equation*}
  \int_{S^c\setminus(3Q)}\frac{\ell(Q)^\alpha\abs{Q}}{\abs{x-c_Q}^{d+\alpha}}\ud x
  \eqsim\int_{\max\{\ell(Q),\dist(Q,S^c)\}}^\infty\frac{\ell(Q)^\alpha\abs{Q}}{t^{d+\alpha}}t^{d-1}\ud t
  \eqsim\frac{\ell(Q)^\alpha\abs{Q}}{\max\{\ell(Q),\dist(Q,S^c)\}^{\alpha}}.
\end{equation*}

So altogether we have
\begin{equation*}
  \abs{\pair{T\phi_{Q,i}^1}{1_{S^c}\psi_{R,j;S}^2}}\lesssim \Big(1+\frac{\dist(Q,S^c)}{\ell(Q)}\Big)^{-\alpha}\abs{Q},
\end{equation*}
and then
\begin{equation*}
  \abs{\tilde K^{i,j;k}_{R,S}(x,y)}\lesssim \frac{1_R(x)}{\abs{R}} \sum_{\substack{Q\subseteq S\\ \ell(Q)=2^{-k}\ell(S)}}
      \Big(1+\frac{\dist(Q,S^c)}{\ell(Q)}\Big)^{-\alpha}1_Q(y).
\end{equation*}
The number of the cubes $Q$ with $\dist(Q,S^c)=n\ell(Q)$, $(n=0,1,\ldots,2^k-1)$, is $\mathcal{O}(2^{k(d-1)})$, while each of them has measure $\abs{Q}=2^{-kd}\abs{S}$. Recalling also that $\ell(S)=\ell(R)/2$, we have
\begin{equation*}
\begin{split}
  \Norm{\tilde K^{i,j;k}_{R,S}}{L^2(\R^d\times\R^d)}^2
  &\lesssim\frac{1}{\abs{R}}\sum_{n=0}^{2^k-1} 2^{k(d-1)}\cdot (1+n)^{-2\alpha}\cdot 2^{-kd}\abs{S} \\
  &=2^{-k}\sum_{n=1}^{2^k}n^{-2\alpha}
  \eqsim 2^{-2k\alpha},\qquad \alpha<\frac12.\qedhere
\end{split}
\end{equation*}
\end{proof}

With the help of Lemma \ref{lem:tildeK}, the last sum on the right of \eqref{eq:nestedPart} can now be estimated by}
\begin{equation*}
\begin{split}
  \sum_{k=0}^\infty &\sum_R\sum_{\substack{S\subseteq R \\ \ell(S)=\ell(R)/2}}
     \sum_{j}\abs{\pair{T\mathbb{D}_S^{b_1,k}f}{1_{S^c}\psi_{R,j;S}^{2}}\ave{\mathbb{D}_{R,j}^{b_2}g}_{R_j}} \\
   &\lesssim  \sum_{k=0}^\infty\sum_R\sum_{\substack{S\subseteq R \\ \ell(S)=\ell(R)/2}}\sum_{j,i}
    2^{-k\alpha}\Norm{\mathbb{D}_{S,i}^{b_1,k}f}{2}\Norm{\mathbb{D}_{R,j}^{b_2}g}{2} \\
   &\lesssim  \sum_{k=0}^\infty\sum_R\sum_{j,i}
    2^{-k\alpha}\Norm{\mathbb{D}_{R,i}^{b_1,k+1}f}{2}\Norm{\mathbb{D}_{R,j}^{b_2}g}{2}
   \lesssim  \Norm{f}{2}\Norm{g}{2},
\end{split}
\end{equation*}
which completes this part of the estimate.

\subsection{The paraproduct}
The other part in \eqref{eq:nestedPart}, which still requires attention, is
\begin{equation*}
\begin{split}
  \sum_R\ &\abs{\pair{T\mathbb{D}_R^{b_1}f}{b_{R^{a,2}}^2}\frac{\ave{g}_R}{\ave{b_{R^{a,2}}^2}_R}}
  \lesssim\sum_R\abs{\pair{(\mathbb{D}_R^{b_1})^2{f}+\omega_R^1\mathbb{D}_{R,0}^{b_1}f}{T^*b_{R^{a,2}}^2}\ave{g}_R} \\
  &\leq\sum_R\abs{\pair{\mathbb{D}_R^{b_1}f}{(\mathbb{D}_R^{b_1})^*T^*b_{R^{a,2}}^2}\ave{g}_R}
  +\sum_R\abs{\pair{\mathbb{D}_{R,0}^{b_1}f}{\omega_R^1T^*b_{R^{a,2}}^2}\ave{g}_R}.
\end{split}
\end{equation*} 
{Using estimates of the type $\sum_R\abs{\pair{\phi_R}{\psi_R}}\leq\sum_R\int\abs{\phi_R}\abs{\psi_R}\leq\int(\sum_R\abs{\phi_R}^2)^{1/2}(\sum_R\abs{\psi_R}^2)^{1/2}$ and then H\"older's inequality, we may continue with}
\begin{equation*}
\begin{split}  
  &\leq\BNorm{\Big(\sum_R\abs{\mathbb{D}_R^{b_1}f}^2\Big)^{1/2}}{L^{s'}}
    \BNorm{\Big(\sum_R  \abs{(\mathbb{D}_R^{b_1})^*T^*b_{R^{a,2}}^2\ave{g}_R}^2\Big)^{1/2}}{L^{s}} \\
  &\qquad+\BNorm{\Big(\sum_R\abs{\mathbb{D}_{R,0}^{b_1}f}^2\Big)^{1/2}}{L^{s'}}
     \BNorm{\Big(\sum_R\abs{\omega_R^1 T^*b_{R^{a,2}}^2 \ave{g}_R}^2\Big)^{1/2}}{L^{s}} \\
   &\lesssim\Norm{f}{L^{s'}}\Norm{g}{L^{s}}\sup_S\abs{S}^{-1/s}
       \BNorm{\Big(\sum_{R\subseteq S} \abs{(\mathbb{D}_R^{b_1})^*T^*b_{R^{a,2}}^2}^2\Big)^{1/2}}{L^{s}} \\
     &\qquad+\Norm{f}{L^{s'}}\Norm{g}{L^{s}}\sup_S\abs{S}^{-1/s}
     \BNorm{\Big(\sum_{R\subseteq S}\abs{\omega_R^1 T^*b_{R^{a,2}}^2 }^2\Big)^{1/2}}{L^{s}}, \\
\end{split}
\end{equation*}
where in the last step we used Proposition~\ref{prop:LP} and the following $L^s$ version of the Carleson embedding theorem from \cite[Theorem 8.2]{HMP}:

\begin{proposition}\label{prop:Carleson}
\begin{equation*}
   \BNorm{\Big(\sum_R \abs{ \theta_R\ave{g}_R}^2\Big)^{1/2}}{L^s}
   \lesssim\Norm{g}{L^s}\sup_S\abs{S}^{-1/s}\BNorm{\Big(\sum_{R\subseteq S} \abs{ \theta_R}^2\Big)^{1/2}}{L^s},\qquad s\in(1,2].
\end{equation*}
\end{proposition}

It remains to estimate the two Carleson norms above.

\begin{lemma}
\begin{equation*}
   \BNorm{\Big(\sum_{R\subseteq S} \abs{(\mathbb{D}_R^{b_1})^*T^*b_{R^{a,2}}^2}^2\Big)^{1/2}}{L^{s}}
   \lesssim \abs{S}^{1/s}.
\end{equation*}
\end{lemma}

\begin{proof}
We reorganise the summation as
\begin{equation}\label{eq:stoppingSplit}
  \sum_{R\subseteq S}=\sum_{\substack{R\subseteq S\\ R^{a,2}=S^{a,2}}}
  +\sum_{\substack{P\subsetneq S\\ P^{a,2}=P}}\sum_{\substack{R\subseteq P\\ R^{a,2}=P}}
  =\sum_{k=0}^\infty\sum_{P\in\mathscr{P}^k(S)}\sum_{\substack{R\subseteq P\\ R^{a,2}=P^{a,2}}},
\end{equation}
where $\mathscr{P}^0(S):=\{S\}$, $\mathscr{P}^1(S)$ consists of the maximal $P\subsetneq S$ with $P=P^{a,2}$, and recursively $\mathscr{P}^{k+1}(S):=\bigcup_{G\in\mathscr{P}^k(S)}\mathscr{P}^1(G)$. Since all $P\in\mathscr{P}^k(S)$ are disjoint for a fixed $k$, we get
\begin{equation*}
  \BNorm{\Big(\sum_{R\subseteq S} \abs{(\mathbb{D}_R^{b_1})^*T^*b_{R^{a,2}}^2}^2\Big)^{1/2}}{L^{s}}
  \leq\sum_{k=0}^\infty\Big(\sum_{P\in\mathscr{P}^k(S)}
    \BNorm{\Big(\sum_{\substack{R\subseteq P\\ R^{a,2}=P^{a,2}}} \abs{(\mathbb{D}_R^{b_1})^*T^*b_{P^{a,2}}^2}^2\Big)^{1/2}}{L^{s}}^s\Big)^{1/s},
\end{equation*}
{and then by the square function estimate from Proposition \ref{prop:LP},
\begin{equation*}
   \lesssim\sum_{k=0}^\infty\Big(\sum_{P\in\mathscr{P}^k(S)}\Norm{1_P T^*b_{P^{a,2}}^2}{L^s}^s\Big)^{1/s}=:A.
\end{equation*}
Next, using H\"older's inequality (recall that $s'>t'$, hence $s<t$) and the assumption of Proposition \ref{prop:babyTb} that $b_Q^2$ is $(\infty,t)$ accretive for $T^*$ on the stopping cubes, 
\begin{equation*}
  \Norm{1_P T^*b_{P^{a,2}}^2}{L^s}
  \leq\abs{P}^{1/s-1/t}\Norm{1_P T^*b_{P^{a,2}}^2}{L^t}\lesssim\abs{P}^{1/s}.
\end{equation*}
Therefore}
\begin{equation*}
  {A\lesssim }\sum_{k=0}^\infty\Big(\sum_{P\in\mathscr{P}^k(S)}\abs{P}\Big)^{1/s}  
  \leq\sum_{k=0}^\infty\Big((1-\tau)^{k-1}\abs{S}\Big)^{1/s} \lesssim \abs{S}^{1/s}.\qedhere
\end{equation*}
\end{proof}

\begin{lemma}
\begin{equation*}
  \BNorm{\Big(\sum_{R\subseteq S}\abs{\omega_R^1 T^*b_{R^{a,2}}^2 }^2\Big)^{1/2}}{L^{s}}
  \lesssim\abs{S}^{1/s}.
\end{equation*}
\end{lemma}

\begin{proof}
This is proven by a similar splitting but, instead of the square function estimate for $(\mathbb{D}_R^{b_1})^*$, using
\begin{equation*}
\begin{split}
  \BNorm{\Big(\sum_{R\subseteq P} \abs{\omega_R^1 T^*b_{P^{a,2}}^2}^2\Big)^{1/2}}{L^{s}}
   &=\BNorm{\Big(\sum_{R\subseteq P} \abs{\omega_R^1}^2\Big)^{1/2} \abs{T^*b_{P^{a,2}}^2}}{L^{s}} \\
   &\leq \BNorm{\Big(\sum_{R\subseteq P} \abs{\omega_R^1}^2\Big)^{1/2}}{L^u}\Norm{1_PT^*b_{P^{a,2}}^2}{L^t},\qquad \frac{1}{s}=\frac{1}{u}+\frac{1}{t}.
\end{split}
\end{equation*}
Then, using a decomposition as in \eqref{eq:stoppingSplit} but in terms of the stopping cubes on the $b_1$-side rather than $b_2$ side,
\begin{equation*}
\begin{split}
  \BNorm{\Big(\sum_{R\subseteq P} \abs{\omega_R^1}^2\Big)^{1/2}}{L^u}
  &\lesssim\BNorm{\Big(\sum_{\substack{R\subseteq P\\ R^a=R}} 1_R \Big)^{1/2}}{L^u}
  \leq\sum_{k=0}^\infty\Big(\sum_{G\in\mathscr{P}^k(P)}\abs{G}\Big)^{1/u} \\
  &\leq\sum_{k=0}^\infty\Big((1-\tau)^{k-1}\abs{P}\Big)^{1/u}\lesssim\abs{P}^{1/u}.
\end{split}
\end{equation*}
Since $\Norm{1_PT^*b_{P^{a,2}}^2}{L^t}\lesssim\abs{P}^{1/t}$, we get the asserted bound $\abs{P}^{1/u}\abs{P}^{1/t}=\abs{P}^{1/s}$.
\end{proof}

\subsection{The diagonal}
It only remains to estimate
\begin{equation*}
\begin{split}
  \sum_R\abs{\pair{T\mathbb{D}_R^{b_1}f}{\mathbb{D}_R^{b_2}g}}
  &\leq\sum_R\sum_{i,j=1}^{2^d} \abs{\pair{T(1_{R_i}\mathbb{D}_R^{b_1}f)}{1_{R_j}\mathbb{D}_R^{b_2}g}} \\
  &\lesssim\sum_R\sum_{i,j:i\neq j}\Norm{1_{R_i}\mathbb{D}_R^{b_1}f}{2}\Norm{1_{R_j}\mathbb{D}_R^{b_2}g}{2}
     +\sum_R\sum_j\abs{\pair{T(1_{R_j}\mathbb{D}_R^{b_1}f)}{1_{R_j}\mathbb{D}_R^{b_2}g}},
\end{split}
\end{equation*}
where the unequal subcubes were estimated by Hardy's inequality, and this part is readily bounded by $\Norm{f}{2}\Norm{g}{2}$. For the final part, we write, as in \eqref{eq:nestRewrite},
\begin{equation*}
\begin{split}
  1_{R_j}\mathbb{D}_R^b f
  =\frac{1_{R_j}b_{R_j^{a}}}{\ave{b_{R_j^a}}_{R_j}}\ave{\mathbb{D}_R^b f}_{R_j}
  +1_{R_j}\Big(\frac{\ave{b_{R^a}}_{R_j}}{\ave{b_{R_j^a}}_{R_j}}b_{R_j^a}-b_{R^{a}}\Big)\frac{\ave{\mathbb{D}_{R,0}^b f}_{R}}{\ave{b_{R^{a}}}_{R}}.
\end{split}
\end{equation*}
and similarly for $1_{R_j}\mathbb{D}_R^{b_2}g$. Thus
\begin{equation*}
\begin{split}
  &\abs{\pair{T(1_{R_j}\mathbb{D}_R^{b_1}f)}{1_{R_j}\mathbb{D}_R^{b_2}g}}  \\
  &\lesssim\Big(\sum_{i,h\in\{0,j\}}\abs{\pair{T(1_{R_j}b^1_{R_i^{a,1}})}{1_{R_j}b^2_{R_h^{a,2}}}}\Big)
     \Big(\sum_{i,h\in\{0,j\}}\abs{\ave{\mathbb{D}_{R,i}^{b_1}f}_{R_i}}\abs{\ave{\mathbb{D}_{R,h}^{b_2}g}_{R_h}}\Big),
\end{split}
\end{equation*}
and it remains to check that the first term is dominated by $\abs{R}$, for then e.g.
\begin{equation*}
  \abs{R}^{1/2}\abs{\ave{\mathbb{D}_{R,i}^{b_1}f}_{R_i}}\lesssim\Norm{\mathbb{D}_{R,i}^{b_1}f}{2},
\end{equation*}
and it is easy to conclude by Proposition~\ref{prop:LP}.

If $i=j$, then
\begin{equation*}
 \abs{ \pair{T(1_{R_j}b^1_{R_i^{a,1}})}{1_{R_j}b^2_{R_h^{a,2}}} }
 \leq \Norm{1_{R_j}T(1_{R_j}b^1_{R_j^{a,1}})}{1}\Norm{b^2_{R_h^{a,2}}}{\infty}
 \lesssim \abs{R_j}\cdot 1\lesssim\abs{R}
\end{equation*}
directly from the definition of strong accretivity (with $Q'=R_j$ and $Q=R_j^{a,1}$)

If $i=0$, then
\begin{equation*}
\begin{split}
    \abs{ \pair{T(1_{R_j}b^1_{R_i^{a,1}})}{1_{R_j}b^2_{R_h^{a,2}}} }
    &= \abs{ \pair{T(1_{R}b^1_{R^{a,1}})}{1_{R_j}b^2_{R_h^{a,2}}} - \pair{T(1_{R\setminus R_j}b^1_{R^{a,1}})}{1_{R_j}b^2_{R_h^{a,2}}} }  \\
    & \lesssim \Norm{1_R T(1_{R}b^1_{R^{a,1}})}{1}\Norm{b^2_{R_h^{a,2}}}{\infty}
        +\Norm{1_{R\setminus R_j}b^1_{R^{a,1}}}{2}\Norm{1_{R_j}b^2_{R_h^{a,2}}}{2} \\
     &\lesssim\abs{R}\cdot 1+\abs{R}^{1/2}\abs{R}^{1/2}\lesssim\abs{R},
\end{split}
\end{equation*}
where we used Hardy's inequality for the second term and the strong accretivity (with $Q'=R$ and $Q=R^{a,1}$) to the first one.

This completes the proof of the ``baby $Tb$ theorem'', Proposition~\ref{prop:babyTb}. Since this was the last missing ingredient, the proofs of the Main Theorem~\ref{thm:main} and Theorem~\ref{thm:Hofmann} are also finished.

\section{Concluding remark}

We have completed the proofs of Theorems~\ref{thm:Hofmann} and \ref{thm:main}, dealing with the boundedness of singular integral operators on $\R^d$ with the Lebesgue measure. Using the dyadic cubes of Christ~\cite{Christ:90} in place of the Euclidean dyadic cubes, these results extend to a metric space with a doubling measure without difficulty. In particular, a Hardy inequality is also valid for Christ's dyadic cubes, as observed by Auscher and Routin \cite{AR}. A version for non-doubling measures is recently due to Martikainen, Mourgoglou and Tolsa \cite{MMT}, but only for antisymmetric kernels.


\begin{thebibliography}{10}

\bibitem{AHMTT}
P.~Auscher, S.~Hofmann, C.~Muscalu, T.~Tao, and C.~Thiele.
\newblock Carleson measures, trees, extrapolation, and {$T(b)$} theorems.
\newblock {\em Publ. Mat.}, 46(2):257--325, 2002.

\bibitem{AR}
P.~Auscher and E.~Routin.
\newblock Local {$Tb$} theorems and {H}ardy inequalities.
\newblock {\em J. Geom. Anal.}, 23(1):303--374, 2013.

\bibitem{AY}
P.~Auscher and Q.~X. Yang.
\newblock B{CR} algorithm and the {$T(b)$} theorem.
\newblock {\em Publ. Mat.}, 53(1):179--196, 2009.

\bibitem{Christ:90}
M.~Christ.
\newblock A {$T(b)$} theorem with remarks on analytic capacity and the {C}auchy
  integral.
\newblock {\em Colloq. Math.}, 60/61(2):601--628, 1990.

\bibitem{DJ:T1}
G.~David and J.-L. Journ{\'e}.
\newblock A boundedness criterion for generalized {C}alder\'on-{Z}ygmund
  operators.
\newblock {\em Ann. of Math. (2)}, 120(2):371--397, 1984.

\bibitem{DJS}
G.~David, J.-L. Journ{\'e}, and S.~Semmes.
\newblock Op\'erateurs de {C}alder\'on-{Z}ygmund, fonctions para-accr\'etives
  et interpolation.
\newblock {\em Rev. Mat. Iberoamericana}, 1(4):1--56, 1985.

\bibitem{Hof:unpubl}
S.~Hofmann.
\newblock A proof of the local {$Tb$} theorem for standard
  {C}alder\'on-{Z}ygmund operators.
\newblock {\em Unpublished manuscript}, 2007.
\newblock arXiv:0705.0840.

\bibitem{Hof:Escorial}
S.~Hofmann.
\newblock Local {$T(b)$} theorems and applications in {PDE}.
\newblock In {\em Harmonic analysis and partial differential equations}, volume
  505 of {\em Contemp. Math.}, pages 29--52. Amer. Math. Soc., Providence, RI,
  2010.

\bibitem{HyMa:localTb}
T.~Hyt{\"o}nen and H.~Martikainen.
\newblock On general local {$Tb$} theorems.
\newblock {\em Trans. Amer. Math. Soc.}, 364(9):4819--4846, 2012.

\bibitem{HMP}
T.~Hyt{\"o}nen, A.~McIntosh, and P.~Portal.
\newblock Kato's square root problem in {B}anach spaces.
\newblock {\em J. Funct. Anal.}, 254(3):675--726, 2008.

\bibitem{HyVa}
T.~Hyt{\"o}nen and A.~V. V{\"a}h{\"a}kangas.
\newblock The local non-homogeneous {$Tb$} theorem for vector-valued functions.
\newblock {\em Glasg. Math. J.}, 57(1):17--82, 2015.

\bibitem{LaMa:localTb}
M.~T. Lacey and H.~Martikainen.
\newblock Local {$Tb$} theorem with {$L^2$} testing conditions and general
  measures: {C}alder\'on-{Z}ygmund operators.
\newblock {\em Ann. Sci. \'Ec. Norm. Sup\'er. (4)}, 49(1):57--86, 2016.

\bibitem{LaVa:perfect}
M.~T. Lacey and A.~V. V{\"a}h{\"a}kangas.
\newblock The perfect local {$Tb$} theorem and twisted martingale transforms.
\newblock {\em Proc. Amer. Math. Soc.}, 142(5):1689--1700, 2014.

\bibitem{MMT}
H.~Martikainen, M.~Mourgoglou, and X.~Tolsa.
\newblock Improved {C}otlar's inequality in the context of local {$Tb$}
  theorems.
\newblock {\em J. Funct. Anal.}, 274(5):1255--1275, 2018.

\bibitem{McMe}
A.~McIntosh and Y.~Meyer.
\newblock Alg\`ebres d'op\'erateurs d\'efinis par des int\'egrales
  singuli\`eres.
\newblock {\em C. R. Acad. Sci. Paris S\'er. I Math.}, 301(8):395--397, 1985.

\bibitem{NTV:Duke}
F.~Nazarov, S.~Treil, and A.~Volberg.
\newblock Accretive system {$Tb$}-theorems on nonhomogeneous spaces.
\newblock {\em Duke Math. J.}, 113(2):259--312, 2002.

\bibitem{NTV:Vitushkin}
F.~Nazarov, S.~Treil, and A.~Volberg.
\newblock The {$Tb$}-theorem on non-homogeneous spaces that proves a conjecture
  of {V}itushkin.
\newblock {\em Centre de Recerca Matem\`atica (Barcelona) Preprint}, 519:1--84,
  2002.
\newblock arXiv:1401.2479.

\end{thebibliography}

\def\cprime{$'$}

\end{document}